\documentclass[letter,11pt]{article}

\usepackage{latexsym,url,comment,graphicx,epstopdf,enumitem,float}
\usepackage{algorithm,algpseudocode}
\usepackage{forest}

\usepackage{afterpage}
\usepackage{amsthm}
\usepackage{booktabs}
\usepackage{supertabular}

\usepackage{url,centernot}
\usepackage{balance}
\usepackage{graphicx,color}
\usepackage{bm}
\usepackage[utf8x]{inputenc}
\usepackage{geometry}
\usepackage{tabularx,multicol,multirow}
\usepackage{blkarray}
\usepackage[normalem]{ulem}
\usepackage{longtable}
\usepackage{float}
\usepackage{pdflscape}
\usepackage[symbol]{footmisc}
\usepackage{mathtools}

\usepackage{amsfonts}
\usepackage{amssymb}
\usepackage{amsmath}

\newtheorem{theorem}{Theorem}
\newtheorem{lemma}{Lemma}

\newtheorem{corollary}{Corollary}

\newtheorem{conjecture}{Conjecture}

\mathtoolsset{showonlyrefs,showmanualtags}

\DeclareMathOperator*{\argmin}{arg\,min}

\DeclarePairedDelimiter{\ceil}{\lceil}{\rceil}
\DeclarePairedDelimiter{\floor}{\lfloor}{\rfloor}

\forestset{
	declare toks={way label}{},
}

\makeatletter
\newcommand{\raisemath}[1]{\mathpalette{\raisem@th{#1}}}
\newcommand{\raisem@th}[3]{\raisebox{#1}{$#2#3$}}
\makeatother

\makeatletter
\newcommand*{\textoverline}[1]{$\overline{\hbox{#1}}\m@th$}
\makeatother

\newcommand{\kls}{\textsc{K$^\textnormal{th}$Largest Subset}}
\newcommand{\lcm}{\textnormal{lcm}}
\newcommand{\myif}{\textnormal{if }}

\newcommand{\longunderline}[1]{\uline{#1\hfill\mbox{}}}

\renewcommand{\times}{\cdot}

\begin{document}
	
	\title{Further Results on an Abstract Model for Branching and its Application
    to Mixed-Integer Programming}
  
	\author{\normalsize Daniel Anderson$^\dagger$ \\ \normalsize Carnegie Mellon University \\ \normalsize dlanders@cs.cmu.edu \and \normalsize Pierre Le Bodic \\ \normalsize Monash University \\ \normalsize pierre.lebodic@monash.edu \and \normalsize Kerri Morgan$^\dagger$ \\ \normalsize Deakin University \\ \normalsize kerri.morgan@deakin.edu.au}
    \date{}
    
    \footnotetext{$^\dagger$Work completed in part while the author was employed at Monash University}
    
	\maketitle
  
  \begin{abstract}
    A key ingredient in branch and bound (B\&B) solvers for mixed-integer programming (MIP) is the selection of branching variables since poor or arbitrary selection can affect the size of the resulting search trees by orders of magnitude. A recent article by Le Bodic and Nemhauser [{\it Mathematical Programming}, (2017)] investigated variable selection rules by developing a theoretical model of B\&B trees from which they developed some new, effective scoring functions for MIP solvers. In their work, Le Bodic and Nemhauser left several open theoretical problems, solutions to which could guide the future design of variable selection rules. In this article, we first solve many of these open theoretical problems. We then implement an improved version of the model-based branching rules in SCIP 6.0, a state-of-the-art academic MIP solver, in which we observe an $11\%$ geometric average time and node reduction on instances of the MIPLIB 2017 Benchmark Set that require large B\&B trees.
  \end{abstract}

	\section{Introduction}
	
	Modern mixed-integer programming (MIP) solvers and many other combinatorial optimisation technologies are driven by the branch and bound (B\&B) method \cite{land1960automatic}. In MIP solvers, B\&B consists in solving a linear program (LP) relaxation of the MIP and recursively splitting the domains of integer variables whose values are fractional in the solution to the relaxation. The choice of fractional variable to branch on can drastically affect the size of the resulting search trees, potentially leading to orders of magnitude variations in solving time. Modern branching rules prefer branching on variables that most improve the dual bound at the created children, and thus, for the purpose of ranking them, reduce candidate variables to couples of values corresponding to their estimated dual gap improvements. The analysis and application of these couples of improvements, or so called \emph{branching tuples} for n-ary branching, is studied extensively in the context of Satisfiability by Kullman \cite{kullmann2009fundaments}, where it is shown that their quality can be essentially described by a single corresponding real number called the $\tau$-value. Kullman shows that various analytic properties of the $\tau$-value can help theoretically explain why certain branching rules work better than others. Le Bodic and Nemhauser \cite{LeBodic2017} study an abstract model of B\&B trees, and characterize the asymptotic growth rate of the trees resulting from a single variable by the \emph{ratio} value $\varphi,$ which is related to the $\tau$-value of Kullman \cite{kullmann2009fundaments}. Le Bodic and Nemhauser demonstrate experimentally that branching rules incorporating the $\varphi$-value and the abstract tree model perform on average better than the default rules implemented in the MIP solver SCIP 3.1.1 \cite{achterberg2009scip}. As part of their experimental work, they use their \emph{general variable branching} (GVB) model as a simulation problem in order to tune the branching rules that they subsequently implement in SCIP. In their analysis, Le Bodic and Nemhauser leave open several problems related to the model:
  \begin{itemize}[leftmargin=*]
    \item Is there a closed-form formula for computing the ratio $\varphi$?
    \item They conjecture that all instances of the Multiple Variable Branching (MVB) problem admit a single variable that is optimal to branch on for all sufficiently large values of the dual gap.
    \item They pose the question of whether GVB admits tighter hardness results, and whether it admits an approximation algorithm.
  \end{itemize}
  
  \noindent The theoretical contributions of this paper are as follows:
  \begin{itemize}[leftmargin=*]
    \item We show that there exists no closed form formula for the ratio $\varphi$ in general (Section \ref{sec:svb}).
    \item We resolve the \emph{MVB conjecture}, showing that it does not hold in general (Section \ref{sec:mvb}).
    \item We show that GVB admits a Karp reduction from the complement of the classic \kls{} problem, and that it is consequently also PP-hard under polynomial time Turing reductions (Section \ref{sec:gvb}).
    \item We show that Proposition 3 of Le Bodic and Nemhauser \cite{LeBodic2017} is actually false (Section \ref{sec:falseprop}).
    \item We give an analysis of the GVB simulation problem, showing that it can be solved in expected sub-exponential time in $n$ for scoring functions that respect dominance, which improves on the naive exponential time bound (Section \ref{sec:smaller_dp_state_space}).
  \end{itemize}

  \noindent Our practical contributions are:
  \begin{itemize}[leftmargin=*]
    \item an improved implementation of the model-based branching rules in the modern academic MIP solver SCIP 6.0 \cite{GleixnerEtal2018OO} (Section \ref{sec:compute_phi}), now released as part of SCIP 7 \cite{gamrath2020a},
    \item experimental results that demonstrate an $11\%$ geometric average speedup and tree size reduction for problems in the MIPLIB 2017 Benchmark Set \cite{miplib2017} that required large B\&B trees (Section \ref{sec:mip_tests}).
  \end{itemize}
	
	\noindent These practical results were not simply performed for the sake of reproducing the results of \cite{LeBodic2017}.
	Indeed, the experiments of \cite{LeBodic2017} were performed in conditions that best verify the developed theory.
	For instance, measures were taken to reduce performance variability (e.g. a bound cutoff was provided), but this in turn obfuscated the extent to which the improvements found would extend to real-world use cases.
	By contrast, this paper uses SCIP ``as is''.
	Furthermore, we use the most recent state-of-the-art test set, MIPLIB 2017 \cite{miplib2017}, and the latest (as of the time of writing) SCIP major version, 6.0 \cite{GleixnerEtal2018OO}, and with the help of the SCIP team, we ran the experiments in the conditions that are used for the development of SCIP.
	Finally, we have improved the implementation of the new branching rules. 
	To the extent that our results are comparable to those of \cite{LeBodic2017} on instances that require large trees, the new implementation significantly improves performance (from $5\%$ in \cite{LeBodic2017} to $11\%$ improvement in geometric mean over SCIP's default).
	
	The practical significance of these results underlines the importance of further foundational research in the line of \cite{LeBodic2017}.
	Hence we address some of the questions left unanswered by Le Bodic and Nemhauser \cite{LeBodic2017}.
	In particular we prove that there is in general an unavoidable overhead to using the new branching rules, as numerical methods must be invoked for the computation of the ratio $\varphi$.
	We further prove that the MVB and GVB problems, essential to the comparison and calibration simulations of the branching rules, are more complex than \cite{LeBodic2017} revealed.
	Nevertheless, we show how these simulations can be improved for GVB via an interesting analysis of the variable space.
	These new simulations confirm the experimental results found in \cite{LeBodic2017} and in this paper.
	
	\subsection{State-of-the-art branching strategies}
	
	The branching strategy employed by the MIP solver SCIP \cite{GleixnerEtal2018OO} is a strategy called \emph{hybrid branching} \cite{achterberg2009hybrid}, which consists of \emph{reliability pseudocost branching} \cite{achterberg2005a} augmented with several techniques from satisfiability and constraint programming. Reliability pseudocost branching consists in initially performing \emph{strong branching} \cite{applegate1995finding}, where candidate variables are branched on in order to measure the resulting change in the dual gap. Since strong branching is computationally expensive, it is substituted with \emph{pseudocost branching} \cite{benichou1971experiments} once enough is known about the historical dual bound changes that have been exhibited by the candidates in order to estimate future changes.
	Techniques from satisfiability and constraint programming are used to further improve the estimates, yielding the state-of-the-art hybrid strategies. For a thorough description, the PhD thesis by Achterberg \cite{achterberg2007thesis} provides discussion and benchmarks of these branching rules as used in SCIP.
  
  Given the (estimated) dual bound changes as measured by the branching rule, candidate variables are ranked based on a \emph{scoring function}, and the best scoring variable is then branched on. The default scoring function used by SCIP is the \texttt{{product}} {function}, given by
	\begin{equation}
	\textnormal{score}(l,r) = \max(\epsilon, l) \times \max(\epsilon, r),
	\end{equation}
	where $(l,r)$ are the dual bound changes from branching left (downward) and right (upward) respectively and $\epsilon = 10^{-6}$ is used to provide useful scores for variables with $\min(l,r) = 0$. A higher product score indicates a more favourable variable to branch on. Achterberg \cite{achterberg2007thesis} emphasizes the importance of using a good scoring function by demonstrating that the \texttt{product} function outperforms the previously standard \emph{weighted sum} scoring function by more than $10\%$ on experimental benchmarks.
	
	\subsection{Related work}
	
	A recent trend in the design of branching heuristics has been the application of machine learning to variable selection in MIP solvers. Khalil et al.\ \cite{khalil2016learning} devise a machine learning framework for branching and show that it can produce search tree sizes in line with commercial solvers, although the running time overhead is high. Alvarez et al.\ \cite{alvarez2017machine} similarly develop a machine learning approximation to strong branching, and demonstrate that it produces promising results when combined with the heuristics and separating cuts of CPLEX 12.2, although it performs less well when these are disabled. Finally, in \cite{balcan2018learning}, Balcan et al.\ devise a machine learning framework to learn a nearly-optimal linear combination of branching rules for a given distribution of instances. Although such learning-based approaches show great promise for solving problems from known distributions, they tend to lack in generality and applicability to arbitrary problem instances.
	
	\section{Le Bodic and Nemhauser's problems}\label{sec:open}
	
	In this section, we address open problems on Le Bodic and Nemhauser's abstract B\&B model \cite{LeBodic2017}. The  model consists of \emph{variables} with known, fixed integer gains that model the dual gap changes that occur when branching on a variable. A variable $x$ is thus represented by a pair of positive integers $(l,r)$ with $1 \leq l \leq r$. A \emph{branch and bound tree} (B\&B tree) is a vertex-weighted, full binary tree, where each internal node is associated with a variable $(l,r)$, such that a vertex with weight $g$ has children of weight $g + l$ and $g + r$. A B\&B tree is said to close a gap of $G$ if all of its leaves have weight at least $G$, where the root node has weight $g = 0$.
	
	Note that this model is not a perfect fit for what can be observed in the B\&B, where dual bound changes $l$ and $r$ are not fixed or known in advance, and can be 0.
	Despite these limitations, many insights and improvements can be established based on this model.

	\subsection{The Single Variable Branching problem}\label{sec:svb}
	
	The \textsc{Single Variable Branching} (SVB) problem models a B\&B tree consisting of a single variable $(l,r)$ branched on at every internal node. See Figure \ref{fig:svb2_5} for an example of an SVB tree corresponding to the variable $(2,5)$. The size of the smallest SVB tree that closes a gap of $G$ can be readily expressed by the $r^\textnormal{th}$-order linear recurrence
	\begin{equation}\label{eqn:svb_recurrence}
	t(G) = \begin{cases}
	1 & \myif G \leq 0,  \\
	1 + t(G-l) + t(G-r) & \myif G > 0. \\
	\end{cases}
	\end{equation}
	The asymptotic growth rate of the solution to the recurrence \eqref{eqn:svb_recurrence} is referred to as the \emph{ratio} $\varphi$, defined as
	\begin{equation}
	\varphi = \lim_{G \to \infty} \left(\frac{t(G+l)}{t(G)}\right)^\frac{1}{l},
	\end{equation}
	whose value can be shown to be the unique root greater than $1$ of the trinomial
	\begin{equation}\label{eqn:characteristic_polynomial}
	p(x) = x^r - x^{r-l} - 1.
	\end{equation}
	By definition, the ratio $\varphi$ engenders the useful approximation
	\begin{equation}\label{eqn:svb_approx}
	t(G) \approx t(\tilde{G})\varphi^{G - \tilde{G}}, \qquad \tilde{G} \leq G.
	\end{equation}
  
	Le Bodic and Nemhauser pondered the existence of a closed-form solution to $\varphi$, as this would be of great practical interest since $\varphi$ has been used as an ingredient in constructing effective variable selection rules for MIP. We resolve this question in the negative.
  
  \begin{figure}[H]
      \centering
      \begin{forest}
        for tree={
          draw,
          circle,
          minimum size=2em, 
          inner sep=1pt
        },
        [0,
        [2,edge label={node[midway,above left] {2}},
        [4,edge label={node[midway,  left] {2}},
        [6,edge label={node[midway,  left] {2}},
        ]
        [9,edge label={node[midway, right] {5}},
        ]
        ]
        [7,edge label={node[midway, right] {5}},
        ]
        ]
        [5,edge label={node[midway, above right] {5}},
        [7,edge label={node[midway, left] {2}},
        ]
        [10,edge label={node[midway, right] {5}},
        ]
        ]
        ]
      \end{forest}
      \caption{A minimal SVB tree closing the gap $G = 6$ using the variable $(2,5)$. The edge labels represent the dual bound change and the node labels indicate the total gap closed at a particular node.}\label{fig:svb2_5}
    \end{figure}
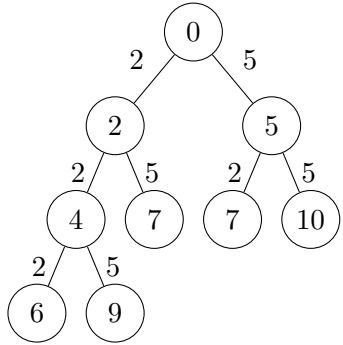
	
	\begin{theorem}\label{thm:no_closed_form}
		There is no closed-form formula for the positive root of the characteristic trinomial \eqref{eqn:characteristic_polynomial} in general.
	\end{theorem}
	
  \noindent We prove this theorem and provide an analysis of the algebraic characteristics of the trinomial $f(x) = x^r - x^{r-l} - 1$. When $r=l$, the polynomial is $f(x)=x^{r}-2$, which has the unique real root $x = 2^{\frac{1}{r}}$.  When $l=0$, the polynomial is $f(x)=-1$ which has no roots. Without loss of generality, we subsequently consider the case where $0 < l < r$.
    
    
    \subsubsection{Solvability by radicals}
    
    We wish to answer the question: for which values of $r$ and $l$, is $f(x)$ solvable by radicals, that is by a finite number of field operations and the taking of $n$-th roots. Let $r=k_{1}d$ and $l=k_{2}d$ where $d=\gcd(r,l)$ and $k_{1}, k_{2}\in \mathbb{N}$.  We first note the following fact.
    \begin{lemma}
      If $d>1$, then $f(x)=x^{k_{1}d}-x^{(k_{1}-k_{2})d}-1$ is solvable by radicals if and only if $F(X)=X^{k_{1}}-X^{k_{1}-k_{2}}-1$ is solvable by radicals.
    \end{lemma}
    \begin{proof}
      We note that $\alpha$ is a root of $F(X)$ if and only if $\alpha^\frac{1}{d}$ is a root of $f(x)$.  
    \end{proof}
    
    \noindent For the remainder of the proof, it suffices to  consider only cases where $\gcd(r, l)=1$ and $0<l<r$.  
    We use the following theorem in \cite{ljunggren1960irreducibility} which states:
    \begin{lemma}[Ljunggren, 1960]\label{thm1}
      If $n=n_{1}d$, $m=m_{1}d$, $\gcd(n_{1}, m_{1})=1$, $n\geq 2m$, then the polynomial 
      \begin{align*}
      g(x)=x^{n}+\epsilon x^{m}+\epsilon ' \text{, }&\epsilon = \pm 1\text{, } \epsilon ' = \pm 1,
      \end{align*}
      is irreducible over $\mathbb{Q}$, except if $n_{1}+m_{1}\cong 0 \pmod 3$ and one of the following cases is true:
      \begin{enumerate}[leftmargin=18pt]
        \item $n_{1}$, $m_{1}$ both odd, $\epsilon =1$;
        \item$n_{1}$ even, $\epsilon ' =1$;
        \item $m_{1}$ even, $\epsilon ' = \epsilon$.
      \end{enumerate}
      If the polynomial $g(x)$ is reducible over $\mathbb{Q}$, then $g(x) = (x^{2d}+(\epsilon)^{m} (\epsilon ')^{n}x^{d}+1) g'(x)$, where $g'(x)$ is irreducible over $\mathbb{Q}$.
    \end{lemma}
    
    \noindent We will also use the following.
    \begin{lemma}\label{lem1}
      $2r-l \cong 0 \pmod 3$ if and only if $r+l \cong 0 \pmod 3$.
    \end{lemma}
    \begin{proof}
      It is clear  that $2r-l \cong 0 \pmod 3$ if and only if  one of the following conditions holds:
      \begin{align*}
      r \cong 0 \pmod 3 \text{  and  } & l \cong 0 \pmod 3,   \\
      r \cong 1 \pmod 3 \text{  and  } & l \cong 2 \pmod 3, \text{ or} \\
      r \cong 2 \pmod 3 \text{  and  } & l \cong 1 \pmod 3. 
      \end{align*}
      These are precisely the cases where $r+l \cong 0 \pmod 3$.
    \end{proof}
    \begin{lemma}\label{lem:irreducability}
      If $r$ and $l$ are co-prime and $0<l < r$, then $f(x)=x^{r}-x^{r-l}-1$ is irreducible over $\mathbb{Q}$ apart 
      from the case where both $r$ and $l$ are odd and $r+l \cong 0 \pmod 3$. 
      In the latter case $f(x)=g(x)(x^{2}-x+1)$ where $g(x)$ is irreducible over $\mathbb{Q}$.  
    \end{lemma}
    \begin{proof}
      Assume $\gcd(r,l)=1$.  We deal with two cases, the case where $r/2 \leq l$, and where $l < r/2$. From now on, when we say that a polynomial is reducible/irreducible, we mean that it is reducible/irreducible over $\mathbb{Q}$.
      
      \begin{itemize}[leftmargin=*]
        \item \textbf{Case 1 ($ r/2 \leq l$):} The proof follows from Lemma \ref{thm1}.    As the coefficients of the non-leading terms of $f(x)$ are $-1$, the
        only case where $f(x)$ can be reducible is the third case, in which it follows that $2r-l \cong 0 \mod 3$ and  $r-l$ is even.  
        By Lemma \ref{lem1}, if $2r-l \cong 0 \pmod 3$, then $r+l \cong 0 \pmod 3$.   As $\gcd(r,l)=1$ and $r-l$ is even,  it follows that $r$ and $l$ must be odd.  
        Thus $f(x)$ is irreducible except when $r$ and $l$ are odd, and  $r+l \cong 0 \pmod 3$. If $f(x)$ is reducible, then $f(x)=(x^{2}-x-1)g(x)$ where $g(x)$ is an irreducible polynomial.
        
        \item \textbf{Case 2 ($l < r/2$):} Consider the polynomial $h(x)=-x^{r}f(1/x)=x^{r}+x^{l}-1$. We will use the fact that $f$ is reducible over $\mathbb{Q}$ if and only if $h$ is reducible over $\mathbb{Q}$. To see this, note that if $h$ is reducible over $\mathbb{Q}$, then there exists a polynomial $h' \in \mathbb{Q}[x]$ of degree $k<r$ that divides $h$. Then observe that $x^k h'(1/x)$ is a polynomial in $\mathbb{Q}[x]$ that divides $f$, so $f$ must be reducible over $\mathbb{Q}$. Similarly, if $f$ is reducible over $\mathbb{Q}$ with divisor $f' \in \mathbb{Q}[x]$ of degree $k'$, then the polynomial $x^{k'} f'(1/x)$ is a divisor of $h$, so $h$ is also reducible over $\mathbb{Q}$.
        By Lemma \ref{thm1}, and noting that $\epsilon=1$ and $\epsilon ' = -1$ in $h(x)$, the polynomial $h(x)$ 
        (and thus the polynomial $f(x)$) is irreducible when $0<l\leq r/2$ apart from the case where both $r$ 
        and $l$ are  odd and $r+l \cong 0\pmod 3$.  If $h(x)$ is reducible, then $h(x)=(x^{2}-x+1)g'(x)$ where $g'(x)$ is an irreducible polynomial.  
        As $f(x)=-x^{r}h(1/x)$,  it follows that if $f(x)$ is reducible, then 
        \begin{align*}
        f(x)&=-x^{r}  \left( \frac{1}{x^{2}} -\frac{1}{x} +1\  \right) g' \left ( \frac{1}{x} \right)\\
        &= -(x^{r}-x^{r-1}+x^{r-2}) g' \left ( \frac{1}{x} \right)\\
        &= -(x^2 - x + 1)x^{r-2} g'\left(\frac{1}{x}\right)\\
        &=(x^2-x+1)g(x)
        \end{align*}
        where $g(x)$ is the irreducible polynomial $-x^{r-2}g'(\frac{1}{x})$.
      \end{itemize}
     
    \end{proof}
    

    \subsubsection{Galois group of $f(x)$} We now show that when $f(x)$ is irreducible over $\mathbb{Q}$, it is not solvable by radicals for $r\geq 5$.  We will use the following theorem in \cite{osada1987galois}.
    \begin{lemma}[Osada, 1987]\label{thm2}
      Let $f(X) = X^{n} + aX^{m} + b$ be a polynomial of integer
      coefficients, that is, $f(X) \in \mathbb{Z}[X]$. Let $a = a_{0}c^{n}$ and $b = b_{0}^{m}c^{n}$, with $a_0, b_0, c \in \mathbb{Z}$. Then the Galois
      group of $f(X)$ is isomorphic to the symmetric group $S_{n}$ of degree n if the following
      conditions are satisfied:
      \begin{itemize}[leftmargin=12pt]
        \item $f(X)$ is irreducible over $\mathbb{Q}[X]$,
        \item $\gcd(a_{0}c(n-m)m, nb_{0})=1$
      \end{itemize}
    \end{lemma}
    
    \begin{lemma}
      If $f(x)=x^{r}-x^{r-l}-1$ is an irreducible polynomial in $\mathbb{Q}[x]$ and $\gcd(r,l)=1$, 
      then $f(x)$ has Galois group $S_{r}$.
    \end{lemma}
    \begin{proof}
      We consider two cases based on the parity of $r$.
      
      \begin{itemize}[leftmargin=*]
        \item \textbf{Case 1 ($r$ is even):} Since $\gcd(r,l) = 1$, we know that $l$ is odd. We apply Lemma~\ref{thm2}, with $a_{0}=b_{0}=-1$ and $c=1$. As $f(x)$ is irreducible, by Lemma \ref{thm2}, $f(x)$ has Galois group $S_{r}$ if 
        \begin{align*}
        \gcd(-1\times(r-r+l)(r-l), -r) = \gcd(l(l-r),-r) = \gcd(l^{2}, r)
        \end{align*}
        equals one, which must be true since $\gcd(l,r) = 1$.
        
        \item \textbf{Case 2 ($r$ is odd):} We apply Lemma~\ref{thm2}, with $a_{0}=b_{0}=1$ and $c=-1$. As $f(x)$ is irreducible, by Lemma \ref{thm2}, $f(x)$ has Galois group $S_{r}$ if 
        \begin{equation}
        \gcd(-1\times(r-r+l)(r-l), r) = \gcd(l(l-r),r) = \gcd(l^{2}, r)
        \end{equation}
        equals one, which must be true since $\gcd(l,r) = 1$.
      \end{itemize}
      
    \end{proof}

    \subsubsection{Consequences for the characteristic trinomial}  The Galois group $S_{n\geq 5}$ is not solvable by radicals.  
    This means that if $f(x)$ is irreducible over $\mathbb{Q}$:
    \begin{itemize}[leftmargin=*]
      \item If $\gcd(r, l)=1$ and $r\geq 5$ then the polynomial is \textbf{not} solvable by radicals,
      \item If $\gcd(r, l)=d$ and $r/d \geq 5$ then the polynomial is \textbf{not} solvable by radicals.
    \end{itemize}
    
    \noindent These points allow us to conclude Theorem~\ref{thm:no_closed_form}. Since any polynomial that factorizes into factors of degree at most four is solvable by radicals, we have, however,
    \begin{itemize}[leftmargin=*]
      \item If $r/\gcd(r,l) \leq 4$ then the polynomial \textbf{is  solvable} by radicals.
    \end{itemize}
    
    \noindent The remaining cases are where $f(x)$ is reducible, that is, according to Lemma~\ref{lem:irreducability}, where both $r$ and $l$ are odd and $r+l \cong 0 \pmod 3$. The factor $x^{2}-x+1$ is not solvable by radicals. The irreducible factor $g(x)$ has degree 
    $r-2$.  If $r <7$ then it is solvable by radicals.  Cases where $f(x)$ is reducible and $r\geq 7$ remain open. The smallest such case is $r=7$ and $l=5$. 
    Here $f(x)=(x^2-x+1)(x^5 + x^4 - x^2 - x - 1)$ where the quintic has Galois group $S_{5}$ and so is not solvable by radicals.

	\subsection{The Multiple Variable Branching problem}\label{sec:mvb}
	
	The \textsc{Multiple Variable Branching} (MVB) problem models a B\&B tree consisting of a set of variables $(l_i, r_i)_{1\leq i \leq n}$ each of which may be branched on an arbitrary number of times. The size of the smallest MVB tree that closes a gap of $G$ can be expressed as the solution to the following non-linear recurrence.
	\begin{equation}
	t(G) = \begin{cases}
	1 & \myif G \leq 0, \\
	1 + \min\limits_{1 \leq i \leq n} (t(G-l_i) + t(G-r_i))  & \myif G > 0. \\
	\end{cases}
	\end{equation}
	The ratio of an MVB tree is defined similarly to that of an SVB tree, such that
	\begin{equation}\label{eqn:mvb_ratio_def}
	\varphi = \lim_{G \to \infty} \left( \frac{t(G+z)}{t(G)} \right)^\frac{1}{z},
	\end{equation}
	where $z = \lcm_{1\leq i \leq n}(l_i, r_i)$. Le Bodic and Nemhauser showed that the ratio for an MVB tree is related to the ratios of the constituent variables such that
	\begin{equation}\label{eqn:mvb_svb_ratio}
	\varphi = \min_{1 \leq i \leq n} \varphi_i,
	\end{equation}
	where $\varphi_i$ is the ratio of the single variable $(l_i, r_i)$.	This interesting result spurred the following specious conjecture.
	
	\begin{conjecture}[Le Bodic and Nemhauser, 2017 \cite{LeBodic2017}]
		For each instance of MVB, there exists a gap $H$ such that for all gaps greater than $H$, variable $i = \argmin_j \varphi_j$ is always optimal to branch on at the root node.
	\end{conjecture}
	
	\begin{theorem}\label{thm:mvb_conjecture_wrong}
		The MVB conjecture is false.
	\end{theorem}
	
  \begin{proof}
    Consider the instance of MVB consisting of the variables $(2,4)$ and $(3,3)$, whose ratios, to five significant digits are $1.27202$ and $2^{\frac{1}{3}}=1.25992$ respectively. Such an MVB tree closing a gap of $G = 8$ is depicted in Figure~\ref{fig:tree_G8}. According to the MVB conjecture, it should be the case that $(3,3)$ is always branched on for all $G$ above some threshold. We prove in Appendix 1 that a closed-form solution to this instance is given by
    \begin{equation}\label{eqn:closed_form_mvb}
    t(G) = \begin{cases}
    2 \cdot 4^k - 1 & \textnormal{if } G = 6k, \\
    \frac{4}{3} \left(2 \cdot 4^k + 1 \right) - 1 & \textnormal{if } G = 1 + 6k, \\
    \frac{2}{3} \left(5 \cdot 4^k + 1 \right) - 1 & \textnormal{if } G = 2 + 6k, \\
    4 \cdot 4^k - 1 & \textnormal{if } G = 3 + 6k, \\
    \frac{2}{3} \left(8 \cdot 4^k + 1 \right) - 1 & \textnormal{if } G = 4 + 6k, \\
    \frac{4}{3} \left(5 \cdot 4^k + 1 \right) - 1 & \textnormal{if } G = 5 + 6k. \\
    \end{cases}
    \end{equation}
    It then suffices to observe from \eqref{eqn:closed_form_mvb} that for any gap of the form $G = 2 + 6k$, branching on $(2,4)$ yields a strictly smaller tree than branching on $(3,3)$, despite the fact that $\varphi(3,3) < \varphi(2,4)$.
  \end{proof}
  
  \begin{figure}[h]
        \centering
        \begin{forest}
        for tree={
          draw,
          circle,
          minimum size=2em, 
          inner sep=1pt
        },
          [0,
          [2,edge label={node[midway, above left] {2}},
            [5,edge label={node[midway, above left] {3}},
              [8,edge label={node[midway, left] {3}},
              ]
              [8,edge label={node[midway, right] {3}},
              ]
            ]
            [5,edge label={node[midway, above right] {3}},
              [8,edge label={node[midway, left] {3}},
              ]
              [8,edge label={node[midway, right] {3}},
              ]
            ]
          ]
          [4,edge label={node[midway, above right] {4}},
            [6,edge label={node[midway, left] {2}},
              [8,edge label={node[midway, left] {2}},
              ]
              [10,edge label={node[midway, right] {4}},
              ]
            ]
            [8,edge label={node[midway, right] {4}},
            ]
          ]
          ]
        \end{forest}
        \caption{A minimal MVB tree for the variables $(2,4)$ and $(3,3)$ closing the gap $G = 8$.}\label{fig:tree_G8}
      \end{figure}
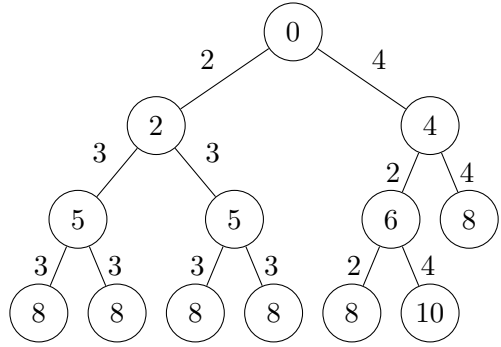
  
  \noindent This result may seem counter-intuitive.
  In the counterexample used in the proof of Theorem \ref{thm:mvb_conjecture_wrong}, we can verify that, for $i \in \{0, \dots, 5\}$,
  \begin{equation}
   \lim_{G \to \infty} \left( \frac{t(G+z)}{t(G)} \right)^\frac{1}{z} = \lim_{k \to \infty} \left( \frac{t(6k+i+12)}{t(6k+i)} \right)^\frac{1}{12} = \lim_{k \to \infty} \left( \frac{\alpha_i \times 4^{k+2} + \beta_i}{\alpha_i \times 4^k + \beta_i} \right)^\frac{1}{12} = 2^{\frac{1}{3}},
  \end{equation}
  where $\alpha_i$ and $\beta_i$ are constants corresponding to the case $i$ in \eqref{eqn:closed_form_mvb}.
  This indicates that the per-unit-of-gap mean growth rate is $2^{\frac{1}{3}}$, which means that the size of the tree doubles for every 3 additional units of gap to close.
  However, this growth is not constant for each additional unit of gap.
  Starting with $G=6k$, and going to $G=6k+1$, the tree asymptotically grows by a factor 
  \begin{equation}
    \lim_{k \to \infty} \left( \frac{t(6k+1)}{t(6k)} \right) = \lim_{k \to \infty} \left( \frac{ \frac{4}{3} \left(2 \cdot 4^k + 1 \right)}{2 \cdot 4^k - 1} \right) = \frac{4}{3}.
  \end{equation}
Taking all such unit gap increments, in the asymptotic case, the 6 growth rates are $\frac{4}{3}, \frac{5}{4}, \frac{6}{5}, \frac{4}{3}, \frac{5}{4}, \frac{6}{5}$, which means that the size of the tree grows by factors $\frac{4}{3}, \frac{5}{4}, \frac{6}{5}$, cyclically.
For every three additional units of gaps, the size of the tree increases by a factor $\frac{4}{3} \times \frac{5}{4} \times \frac{6}{5} = 2$, which matches the MVB ratio $\varphi=2^\frac{1}{3}$.

We can now understand why $(3,3)$ is not always branched on at the root node.
While it also leads to a mean per-unit growth rate of $2^\frac{1}{3}$, its asymptotic per-unit growth rates would be $2, 1, 1$, cyclically: if it branches, the size of the tree doubles, and there is no need to branch for the next two increments of gap.
In constrast, the strategy given in \eqref{eqn:closed_form_mvb} progressively grows the tree (but still doubles it for every three units of gaps).
As a result, for cases $G=6k$ and $G=6k+3$, the sizes of either strategy is the same, but for the next steps, i.e. cases $G=6k+1$ and $G=6k+4$, the ``MVB conjecture'' strategy grows the tree by a factor 2,  while the optimal stratey does so by a factor $\frac{4}{3}$, leading to a smaller tree.

Clearly, this intriguing phenomenon can only occur in special circumstances.
While it may not be an artefact of the MVB model, it may be too difficult to detect and exploit in a practical branching rule.
  
	\subsection{The General Variable Branching problem}\label{sec:gvb}
	
	The \textsc{General Variable Branching} (GVB) problem models a B\&B tree consisting of a set of variables $(l_i, r_i)$, each of which may be branched on a maximum of $m_i$ times. Consequently, unlike SVB and MVB, the GVB problem need not necessarily be feasible for all possible inputs. The size of the smallest GVB tree that closes a gap $G$ is given by
	
	\begin{equation}\label{eqn:gvb_recurrence}
	t(G,\textbf{m}) = \begin{cases}
	1 & \myif G \leq 0, \\
	\infty & \myif \mathbf{m} = \mathbf{0},  \\
	1 + \min\limits_{\substack{1 \leq i \leq n\\m_i > 0}} t(G-l_i, \mathbf{m} - \mathbf{I_i}) + t (G-r_i, \mathbf{m} - \mathbf{I_i})  & \textnormal{otherwise}.
	\end{cases}
	\end{equation}
	
	\noindent where $\mathbf{m}$ is the vector of multiplicities and $\mathbf{I_i}$ is an indicator vector on the $i^\textnormal{th}$ element. Le Bodic and Nemhauser showed that GVB is \#P-Hard under polynomial-time Turing reductions via a reduction from a variant of counting Knapsack solutions. They pose the question of whether GVB admits tighter hardness results, and whether it admits an approximation scheme. Here, we provide a Karp reduction to GVB from \kls{}, a well-studied problem in complexity theory, and, using a recent result of Haase and Kiefer \cite{haase2016complexity}, we show that GVB is also PP-Hard under polynomial-time Turing reductions.
	
	\medskip
	\noindent\textbf{Problem:} \kls{} \cite{garey2002computers}\\
	\textbf{Input:} Set $A$ with finite cardinal $N$, size $s(a) \in \mathbb{Z}^+$ for each $a \in A$, positive integers $K$ and $B$.\\
	\textbf{Question:} Are there $K$ or more distinct subsets $A' \subseteq A$ for which the sum of the sizes of the elements in $A'$ does not exceed $B$?
  
  \medskip
  \noindent\textbf{Problem:} \textsc{General Variable Branching}\ \cite{LeBodic2017}\\
  \textbf{Input:} $n$ variables encoded by $(l_i, r_i)$, $i = 1, . . . , n$, an integer $G > 0$, an integer $k > 0$, and a vector of multiplicities $m \in \mathbb{Z}^n_{>0}$. \\
  \textbf{Question:} Is there a B\&B tree with at most $k$ nodes that closes
  the gap $G$, branching on each variable $i$ at most $m_i$ times on each path from the root to a leaf?

	\begin{theorem}\label{thm:reduce_kls_to_gvbc}
	There exists a Karp reduction from \kls{} to the complement of GVB.
	\end{theorem}

  	\begin{proof}
    We will reduce \kls\ to the complement of GVB. The proof closely resembles that of Theorem 8 in \cite{LeBodic2017}. Set $n = N + 1$, and define $C = \sum_{i=0}^N s(a_i)$. Suppose that $B < C$, else the instance is trivial. For each $a_i \in A$, create a variable with left and right gains $(C, C + s(a_i))$ and multiplicity one. Finally, variable $n$ has gains $(C, C)$ and multiplicity one. Set the gap to $G = NC + B + 1$ and the threshold to $k = 2^{N + 1} - 2 + 2K$. Observe that all variables dominate the variable $(C, C)$, and that the variables will be branched on in order of size. Each path from the root to level $N$ therefore corresponds to a subset $A' \subseteq A$, where branching left corresponds to not including some element $a_i$ in the gap, and branching right corresponds to including $a_i$. The gap closed at a particular node at level $N$, corresponding to the subset $A' \subseteq A$ (of right branches taken), is therefore
    \begin{equation*}
    g = NC + \sum_{a_i \in A'} |a_i|.
    \end{equation*}
    If the size of $A'$ exceeds $B$, then $g \geq NC + B + 1 = K$ and hence the node is a leaf. Otherwise, the gap $g$ is strictly less than $K$ and the node will have two children, branching on the variable $(C,C)$. The gap at this level is then
    \begin{equation*}
    g = NC + \sum_{a_i \in A'} |a_i| + C \geq G,
    \end{equation*}
    since $B < C$, and hence the children are leaves. Therefore, the number of subsets whose sum does not exceed $K$ is equal to half of the number of leaves at level $N + 1$, so we have
    \begin{equation*}
    (\textnormal{\# Subsets with sum $\leq B$}) = \frac{1}{2}(t(G) - (2^{N+1} -1)).
    \end{equation*}
    Therefore, if there exist at least $K$ subsets whose sum exceeds $B$, then
    \begin{equation*}
    t(G) \geq 2K + 2^{N+1} - 1 \Leftrightarrow \neg (t(G) \leq 2K + 2^{N+1} - 2).
    \end{equation*}
    We can conclude that the B\&B tree will have at most $2^{N+1} - 2 + 2K$ nodes if and only if there are fewer than $K$ distinct subsets $A' \subseteq A$ whose sum does not exceed $B$, so the answer to \kls\ is YES if and only if the answer to GVB is NO. Finally, observe that the reduction clearly takes polynomial time.
  \end{proof}
	
	\begin{corollary}
		We have the following complexity results for GVB
	\begin{enumerate}[leftmargin=20pt]
		\item GVB is NP-Hard, \#P-Hard and PP-Hard under polynomial-time Turing reductions.
		\item If \kls\ is NP-Hard under Karp reductions, then GVB is coNP-Hard under Karp reductions.
	\end{enumerate}
	\end{corollary}
	
  \noindent Whether \kls\ is NP-Hard under Karp reductions remains a long-standing open problem \cite{homer2011computability}.
	
	\section{Analysis of the GVB simulation problem}\label{sec:gvb_algo}
	
	Given a collection of variables $(l_i, r_i)$ with multiplicities $m_i$, instead of seeking the optimal configuration of a general variable branching tree, we can instead seek to measure the size of the tree that would be built by applying a particular variable selection rule. Formally, given a variable selection rule $f(\mathbf{m},G)$ returning the index of the predicted optimal variable, we can measure
	\begin{equation}\label{eqn:gvb_with_selection_rule}
	t(G,\textbf{m}) = \begin{cases}
	1 & \myif G \leq 0, \\
	\infty & \myif \mathbf{m} = \mathbf{0},  \\
	1 + t(G-l_i, \mathbf{m} - \mathbf{I_i}) + t(G-r_i, \mathbf{m} - \mathbf{I_i}) & \textnormal{otherwise},
	\end{cases}
	\end{equation}
	where $i = {f(\mathbf{m},G)}$. In this way, the GVB problem can be used as a tool to analyse and predict the performance of a given variable selection rule. We refer to this as the \emph{GVB simulation problem}. In \cite{LeBodic2017}, this technique is used to evaluate and compare the \texttt{product} and \texttt{ratio} scoring functions on random collections of variables before testing them on real MIP problems. This situation is only tractable due to the fact that the \texttt{product} and \texttt{ratio} scoring functions do not depend on the dual gap. 
	Indeed, for these two scoring functions, the order in which the variables are branched on (from the root to a leaf) is fixed, and the same variable is branched on at every node of the same depth.
	In this case, the recurrences \eqref{eqn:gvb_recurrence} and \eqref{eqn:gvb_with_selection_rule} can be evaluated in $O(n^2G)$ time and $O(nG)$ space with dynamic programming after pre-computing the variable order. In general, the variables may be branched on in different orders depending on the dual gap along the path, in which case \eqref{eqn:gvb_recurrence} and \eqref{eqn:gvb_with_selection_rule} naively require $O(n2^n G)$ time and $O(2^n G)$ space to consider all subsets of available variables, which is intractable for moderately sized test cases.
  
  However, all scoring functions that we know of respect the notion of \emph{dominance}. For such scoring functions, the state space of the dynamic program can be restricted to those subsets of variables that are \emph{dominance free}. In this section, we provide an analysis of the expected number of dominance free subsets in a collection of random variables. This analysis implies that the expected size of the state space for such a problem is sub-exponential in $n$.
	
	\subsection{Dominated and non-dominated variables}\label{sec:falseprop}
	
	We recall from \cite{LeBodic2017} the notion of dominance. We say that $(l_1, r_1)$ dominates $(l_2, r_2)$ if $l_1 \geq l_2$ and $r_1 \geq r_2$ with at least one of $l_1 > l_2$ or $r_1 > r_2$. We call a set of variables dominance free if it contains no dominated variables. Proposition 3 of \cite{LeBodic2017} claims that for MVB and GVB, it is never strictly optimal to branch on a dominated variable before those that dominate it.
	
	\begin{theorem}\label{thm:lebodic_false_proposition}
		Proposition 3 of \cite{LeBodic2017} is false in the case of the GVB problem. That is, there exists instances of GVB such that it is strictly optimal to branch on a dominated variable.
	\end{theorem}

  \begin{proof}
    Consider the instance of GVB consisting of the variables $(5,6)$, $(9,9)$, $(5,10)$ all with multiplicity one. The optimal solution for this instance with a gap of $G = 15$ contains $9$ nodes as shown in Figure \ref{fig:gvb_counterexample}. At the root node, this tree branches on the variable $(5,6)$, which is dominated, but we note that any tree branching on $(5,10)$ or $(9,9)$ that closes the same gap has at least $11$ nodes.
  \end{proof}
  
  \begin{figure}[h]
    \centering
    \begin{forest}
      for tree={
        draw,
        circle,
        minimum size=2em, 
        inner sep=1pt
      },
      [0,
      [5,edge label={node[midway,above left] {5}},
      [10,edge label={node[midway, left] {5}},
      [19,edge label={node[midway, left] {9}},
      ]
      [19,edge label={node[midway, right] {9}},
      ]
      ]
      [15,edge label={node[midway, right] {10}},
      ]
      ]
      [6,edge label={node[midway,above right] {6}},
      [15,edge label={node[midway, left] {9}},
      ]
      [15,edge label={node[midway, right] {9}},
      ]
      ]
      ]
    \end{forest}
    \caption{A minimal GVB tree for the variables $(5,6), (9,9), (5,10)$ closing the gap $G = 15$.}\label{fig:gvb_counterexample}
  \end{figure}
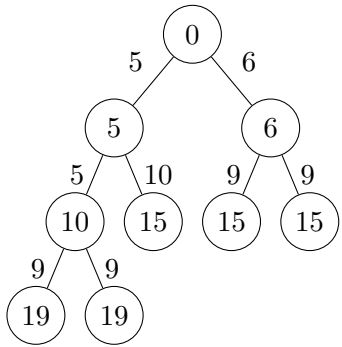	
	
	\noindent This result may seem surprising, similar to Theorem \ref{thm:mvb_conjecture_wrong}.
	Both of these results seem to indicate that branching rules that select variables based on a score that is computed independently of other variables and of the gap are in general not optimal for MVB or GVB.
  We note that similar phenomena have been observed in cutting plane generation. In \cite{owen2001disjunctive}, the authors demonstrate an algorithm, which, if it selects the optimal undominated cut, will never converge, but where a seemingly sub-optimal choice will.
	Again, we do not believe that this is an artifact of the MVB and GVB models, but rather cases that occur so rarely and would be so hard to detect with certainty in an implementation of the B\&B for MIP that they have been (and perhaps should be) ignored in practice.
	For the remainder of the paper we therefore focus on rules that only branch on non-dominated variables.
	
	\subsection{Bounding the expected number of non-dominated subsets}\label{sec:smaller_dp_state_space}
	
	The state space of the naive GVB dynamic program (all possible subsets of variables at every gap) is prohibitively large, but when considering a scoring function that only selects non-dominated variables, the only relevant states are those corresponding to a dominance-free subset of variables. For such a scoring function, we note that any subset of yet-to-be-used variables considered by the algorithm will never contain a variable that dominates a variable that has been used, and that such a subset, of which there should be significantly fewer, can be uniquely identified by its non-dominated variables. Here, we provide an analysis of the expected number of dominance-free subsets in a collection of random variables, which shows that the GVB simulation dynamic program can in fact be solved in sub-exponential time for scoring functions that respect dominance.
	
	\begin{theorem}\label{thm:count_nondominated_subsets}
		Given $n$ pairs $(l_i,r_i)_{1\leq i \leq n}$ such that each $l_i$ is unique, each $r_i$ is unique\footnote{Observe that if $l$ or $r$ contain non-unique values, the chance that they are dominated only increases. Hence analysing the unique case provides an upper bound in the general case}, and $l_i, r_i$ are independent random variables, the expected number of non-dominated subsets is given by
		\begin{equation}\label{eqn:expected_nondominated_subsets}
		\sum_{k=0}^n \frac{1}{k!} \binom{n}{k}.
		\end{equation}
	\end{theorem}
	
  \begin{proof}
    We consider all subsets of size $k$ for some fixed constant $k$. There are a total of $\binom{n}{k}$ such subsets. Take an arbitrary subset $S = \{ v_1, v_2, ..., v_k \}$ and suppose we sort the pairs in descending order of $r$, and then denote the $i^\textnormal{th}$ pair in this sorted order by $(l_i, r_i)$, for $1 \leq i \leq k$.
    This subset is non-dominated if and only if $S$ happens to be sorted in increasing order by $l$, since otherwise we would have $r_i < r_j$ and $l_i < l_j$ for some $i < j$. Since each pair $l_i, r_r$ are independent, each permutation of $l$'s is equally likely, and since there are $k!$ possible permutations,
    \begin{equation*}
    \Pr(S\textnormal{ is non-dominated}) = \frac{1}{k!}.
    \end{equation*}
    Hence the expected number of non-dominated subsets of size $k$ is
    \begin{equation}\label{eqn:nondominated_summand}
    \frac{1}{k!} \binom{n}{k}.
    \end{equation}
    We conclude that the expected number of non-dominated subsets of all sizes is
    \begin{equation}
    \sum_{k=0}^n \frac{1}{k!} \binom{n}{k}.
    \end{equation}
  \end{proof}
  
	\begin{theorem}\label{thm:asymptotic_nondominated_subsets}
		The expected number of non-dominated subsets is sub-exponential in $n$. In particular, a bound on the expected number of subsets of non-dominated variables is
		\begin{equation*}
		\sum_{k=0}^n \frac{1}{k!} \binom{n}{k} = O\left(e^{2\sqrt{n}}\right), \qquad n \to \infty.
		\end{equation*}
	\end{theorem}

  From Theorem~\ref{thm:asymptotic_nondominated_subsets}, it follows that the GVB simulation problem can be solved in expected sub-exponential time in $n$, specifically, in \mbox{$O(\text{poly}(n)e^{2\sqrt{n}}\cdot G)$} time, which significantly improves upon the naive $O(n2^nG)$ bound. We demonstrate the utility of this improved analysis in Appendix 2, where we perform simulations on the \texttt{svts} rule of \cite{LeBodic2017} and compare it to the \texttt{product} and \texttt{ratio} rules. The proof of Theorem~\ref{thm:asymptotic_nondominated_subsets} follows from several Lemmas.
  
  \begin{lemma}\label{lem:summand_max}
    The summand \eqref{eqn:nondominated_summand} satisfies
    \begin{equation}
    \max_{0 \leq k \leq n} \frac{1}{k!}\binom{n}{k} = O\left( \frac{1}{(k^*)!} \binom{n}{k^*} \right),
    \end{equation}
    for
    \begin{equation}
    k^* = \sqrt{n + \frac{1}{4}} - \frac{1}{2}
    \end{equation}
    as $n \to \infty$.
  \end{lemma}
  
  \begin{proof}
    Consider maximizing the expression 
    \begin{equation}\label{eqn:summand}
    \frac{1}{k!}\binom{n}{k} = \frac{n!}{(k!)^2(n-k)!},
    \end{equation}
    over $0 \leq k \leq n$. This is equivalent to minimizing the denominator, $(k!)^2(n-k)!$. We consider a suitable analytic extension, valid for real $n$ and $k$, in terms of the gamma function $\Gamma$, and seek the minimum over $0 \leq k \leq n$ of
    \begin{equation} \label{eqn:analytic_denominator}
    \Gamma(k+1)^2 \Gamma(n-k+1).
    \end{equation}
    We apply the fact that the derivative of the gamma function is given by
    \begin{equation}
    \Gamma'(x+1) = \Gamma(x+1)\psi(x+1),
    \end{equation}
    where $\psi$ is the digamma function, which for a positive integer $x$, satisfies $\psi(x+1) = H_x - \gamma$, where $H_x = \sum_{i=1}^x \frac{1}{i}$ is the $x^\textnormal{th}$ Harmonic number and $\gamma \approx 0.57721567$ is the Euler-Mascheroni constant.
    Differentiating \eqref{eqn:analytic_denominator}, we find
    \begin{equation}
    2(\Gamma(k+1))^2(H_k - \gamma)\Gamma(n-k+1) - (\Gamma(k+1))^2\Gamma(n-k+1)(H_{n-k} - \gamma).
    \end{equation}
    Collecting terms and equating to zero, the minimum is a solution to the equation
    \begin{equation}\label{eqn:harmonic_min}
    2H_k - H_{n-k} = \gamma.
    \end{equation}
    Since $H_n \to \log(n) + \gamma$ as $n \to \infty,$ we seek a solution to
    \begin{equation}
    2 (\log(k) + \gamma) - (\log(n-k) + \gamma) = \gamma.
    \end{equation}
    Note that this substitution is only valid under the assumption that both $k \to \infty$ and $n - k \to \infty$ as $n \to \infty$. This is justified since if this were not the case, $k$ or $n - k$ would be bounded, and hence as $n \to \infty$, Equation~\eqref{eqn:harmonic_min} could not be satisfied, as one of the terms on the left-hand side would be bounded while the other goes to $\infty$. Using simple properties of the logarithm, we then arrive at
    \begin{equation}
    k^2 + k - n = 0.
    \end{equation}
    Therefore, asymptotically, the minimum (and hence the maximum of \eqref{eqn:summand}) occurs at
    \begin{equation}
    k^* = \frac{\sqrt{4n + 1} - 1}{2} = \sqrt{n + \frac{1}{4}} - \frac{1}{2},
    \end{equation}
    as $n \to \infty$.
  \end{proof}
  
  \noindent Note that in Lemmas \ref{lem:summand_behavior} and \ref{lem:approx_max} we are abusing the factorial and binomial coefficient notation by using them with non-integer arguments, but these statements can be written using the gamma function.
  \begin{lemma}\label{lem:summand_behavior}
    Let $k = c\sqrt{n}$ for any real $0 < c \leq \sqrt{n}$. The summand satisfies
    \begin{equation}
    \frac{1}{k!}\binom{n}{k} = O\left( \frac{e^{-c^2}}{ c \sqrt{n - c\sqrt{n}}} \left( \frac{e}{c} \right) ^{2c\sqrt{n}} \right),
    \end{equation}
    as $n \to \infty$.
  \end{lemma}
  
  \begin{proof}
    We have
    \begin{equation}
    \frac{1}{k!}\binom{n}{k} = \frac{n!}{((c\sqrt{n})!)^2(n-c\sqrt{n})!}.
    \end{equation}
    Using Stirling's formula, we obtain the asymptotic expansion as $n \to \infty$,
    \begin{equation}
    \frac{1}{k!}\binom{n}{k} \sim \frac{\sqrt{2\pi n} \left(\frac{n}{e}\right)^n}{\left(\sqrt{2 \pi c\sqrt{n}} \left(\frac{c\sqrt{n}}{e}\right)^{c\sqrt{n}}\right)^2\left(\sqrt{2\pi(n-c\sqrt{n})} \left(\frac{n-c\sqrt{n}}{e}\right)^{n - c\sqrt{n}} \right)}.
    \end{equation}
    Canceling common terms, we find
    \begin{equation}
    \begin{split}
    \frac{1}{k!}\binom{n}{k} &\sim \frac{\left(\frac{n}{e}\right)^n}{2\pi c \left(\frac{c\sqrt{n}}{e}\right)^{2c\sqrt{n}} \sqrt{n - c\sqrt{n}} \left(\frac{n - c\sqrt{n}}{e}\right)^{n - c\sqrt{n}}}  \\
    &=\ \frac{e^{c\sqrt{n}} n^{n - c\sqrt{n}}}{2\pi c^{2c\sqrt{n}+1} (n - c\sqrt{n})^{n - c\sqrt{n}} \sqrt{n - c\sqrt{n}}} \\
    &=\ \frac{e^{c\sqrt{n}}}{2\pi c^{2c\sqrt{n}+1} \sqrt{n - c\sqrt{n}}}\left(\frac{n}{n - c\sqrt{n}}\right)^{n - c\sqrt{n}},
    \end{split}
    \end{equation}
    as $n \to \infty$. Considering the term on the right, we have
    \begin{equation}
    \left(\frac{n}{n - c\sqrt{n}}\right)^{n - c\sqrt{n}} = \left(1 - \frac{c}{\sqrt{n}}\right)^{\sqrt{n}(c - \sqrt{n})},
    \end{equation}
    and using the fact that $\left(1 - c/n \right)^n \to e^{-c}$ as $n \to \infty$, we deduce that
    \begin{equation}
    \left(1 - \frac{c}{\sqrt{n}}\right)^{\sqrt{n}(c - \sqrt{n})} = O\left( e^{c\sqrt{n} - c^2} \right),
    \end{equation}
    as $n \to \infty$. The asymptotic behavior of the summand is therefore
    \begin{equation}\label{eqn:asymptotic_summand}
    \begin{split}
    \frac{1}{k!}\binom{n}{k} &= O\left( \frac{e^{c\sqrt{n}}}{2\pi c^{2c\sqrt{n}+1} \sqrt{n - c\sqrt{n}}} e^{c\sqrt{n} - c^2} \right), \\
    &= O\left( \frac{e^{-c^2}}{c \sqrt{n - c\sqrt{n}}} \left( \frac{e}{c} \right) ^{2c\sqrt{n}} \right),
    \end{split}
    \end{equation}
    as $n \to \infty$.
  \end{proof}
  
  \begin{lemma}\label{lem:approx_max}
    For
    \begin{equation}
    k^* = \sqrt{n + \frac{1}{4}} - \frac{1}{2},
    \end{equation}
    we have
    \begin{equation}
    \frac{1}{(k^*)!}\binom{n}{k^*} = O\left( \frac{1}{(\sqrt{n})!} \binom{n}{\sqrt{n}} \right).
    \end{equation}
  \end{lemma}
  
  \begin{proof}
    We make use of the fact that $(n+\Delta)! \geq n! n^\Delta$. This property holds for integer $\Delta$, but can also be shown, using properties of the Gamma function to be valid for non-integral $\Delta$. 
    Let $\hat{k} = \sqrt{n + \frac{1}{2}}$, and write, using the aforementioned fact
    \begin{equation}
    \begin{split}
    \frac{1}{(k^*)!}\binom{n}{k^*} &= \frac{n!}{\left( \left( \hat{k} - \frac{1}{2} \right)! \right)^2 \left( n - \hat{k} + \frac{1}{2} \right)! }, \\
    &\leq \frac{n!}{\left(\hat{k}!\right)^2 \hat{k}^{-1} \left(n - \hat{k}\right)! (n - \hat{k})^\frac{1}{2} }.
    \end{split}
    \end{equation}
    Now we can write
    \begin{equation}
    \begin{split}
    \frac{1}{(k^*)!}\binom{n}{k^*} &\leq \frac{n!}{(\hat{k}!)^2 (n - \hat{k})!} \cdot \frac{\hat{k}}{\sqrt{n - \hat{k}}}, \\
    &= O\left( \frac{n!}{(\hat{k}!)^2 (n - \hat{k})!} \right),
    \end{split}
    \end{equation}
    which follows from the fact that
    \begin{equation}
    \frac{\sqrt{n + \frac{1}{4}}}{\sqrt{n - \sqrt{n +\frac{1}{4}}}} \to 1,
    \end{equation}
    as $n \to \infty$. Finally, since $\hat{k} \to \sqrt{n}$ as $n \to \infty$, we have
    \begin{equation}
    \frac{1}{(k^*)!}\binom{n}{k^*} = O\left( \frac{n!}{(\sqrt{n}!)^2 (n - \sqrt{n})!} \right) = O\left( \frac{1}{(\sqrt{n})!} \binom{n}{\sqrt{n}} \right).
    \end{equation}
  \end{proof}
  
  \paragraph{Proof of Theorem~\ref{thm:asymptotic_nondominated_subsets}. } By Lemma~\ref{lem:summand_max}, Lemma~\ref{lem:summand_behavior} with $c = 1$, and Lemma~\ref{lem:approx_max}, we bound the maximum term of the sum as $n \to \infty$,
  \begin{equation}
  \frac{1}{(\sqrt{n})!}\binom{n}{\sqrt{n}} = O\left( \frac{e^{2\sqrt{n} - 1}}{2\pi \sqrt{n - \sqrt{n}}}\right).
  \end{equation}
  Then, observe that when $c > e$, \eqref{eqn:asymptotic_summand} is a decaying exponential, hence by Lemma~\ref{lem:summand_behavior} there are at most $\floor{e\sqrt{n}+1}$ terms that contribute to the asymptotic behavior of the sum, and therefore,
  \begin{equation}
  \begin{split}
  \sum_{k=0}^n \frac{1}{k!} \binom{n}{k} &= O \left( (e\sqrt{n}+1) \frac{e^{2\sqrt{n} - 1}}{\sqrt{n - \sqrt{n}}} \right), \\
  &= O\left(e^{2\sqrt{n}}\right),
  \end{split}
  \end{equation}
  as $n \to \infty$.

  \section{Implementation and Experiments}
  
  From their abstract B\&B model, Le Bodic and Nemhauser derive two scoring functions for MIP, the \texttt{ratio} rule, and the \texttt{svts} rule.
  \begin{itemize}[leftmargin=*]
    \item The \texttt{ratio} rule first estimates the height of the branch and bound tree as $\floor{G/l}$, and if this is greater than $10$, score the variables based on their ratio. If the estimated tree height is at most $10$, then the \texttt{product} score is used instead.
    \item The \texttt{svts} rule scores variables based on their single variable tree size $t(G)$. For small gaps, or, more specifically, when $\ceil{G/r} > D$ for some threshold $D$, the value of $t(G)$ is computed exactly using the formula of Le Bodic and Nemhauser~\cite{LeBodic2017}. Otherwise, $t(G')$ is computed for the gap $G' = rD$, and then the approximation formula~\eqref{eqn:svb_approx} is used to obtain an approximate value of $t(G)$ given by $t(G') \cdot \varphi^{G - G'}$. The default value of $D$ is $100$.
  \end{itemize} 
  Note that for sufficiently large gaps, \texttt{ratio} and \texttt{svts} are equivalent. In this section, we describe our improved implementation of these scoring rules and provide performance benchmarks that demonstrate their benefits.
  
  \subsection{Efficient computation of the ratio}\label{sec:compute_phi}
  
  Since the ratio $\varphi$ of all non-dominated candidate variables is computed at every node of height greater than $10$ for the \texttt{ratio} rule, and for any node at which the dual gap is high for \texttt{svts}, computing it efficiently is very important.
  Since we showed that there exists no closed-form formula for $\varphi$, we describe here an improved numerical algorithm for computing it. Similarly to \cite{LeBodic2017}, we first scale the gains of a variable $(l,r)$ to $(1,\frac{r}{l})$, so that the resulting value computed is precisely $\varphi^l$. If $r/l \leq 200$, then we use Laguerre's method \cite{acton1990numerical} to find the root of the scaled trinomial, $x^\frac{r}{l} - x^{\frac{r}{l}-1} - 1.$ This method typically converges within two or three iterations. Otherwise, if $r/l > 200$, then we use the fixed-point method given by the following recurrence:
  \begin{equation}
  f(x) = \left(1 - \frac{1}{x}\right)^{-\frac{l}{r}}.
  \end{equation}
  We make an additional optimisation by caching the value of the ratio for each variable. Since in practice, the gains of a variable, and hence also its ratio do not change significantly between two nodes (as it is often given by pseudocosts), we initialize the method at the variable's most recently computed ratio. Experiments show that this improved fixed-point method converges in roughly half as many iterations as the original fixed-point method given in \cite{LeBodic2017}.

  \subsection{Performance tests}\label{sec:mip_tests}
  
  We implemented the \texttt{ratio} and \texttt{svts} scoring functions in SCIP 6.0 \cite{GleixnerEtal2018OO}. No additional modifications besides those required to perform branching decisions were made. We test the performance of our implementation on the MIPLIB 2017 Benchmark Set, the standard performance benchmarking suite for MIP solvers \cite{miplib2017}. The set contains 240 problem instances representing a diverse range of real-world problems. All experiments were run on a cluster with 48 nodes equipped with Intel Xeon Gold 5122 at 3.60GHz and 96GB RAM. Jobs were run exclusively on a node. To reduce variability, each problem is solved three times, with different random seeds used to permute the input, resulting in a total of 720 instances. Each instance is given a time limit of two hours. To analyze the results, we use a similar methodology to Le Bodic and Nemhauser~\cite{LeBodic2017}. We ignore in our results any problem that is solved in less than one second, solved at the root node without any branching (e.g. via presolving), or that is not solved by any of the scoring functions at all. For each problem, we compute the minimum number of instances $N$ solved by any scoring function, and consider for each scoring function, only their best $N$ times and node counts. This gives scoring functions that solve more instances a fair advantage, and ensures that time limited runs do not
  contribute to the time and node counts.
  
  One noteworthy difference between our experiments and those of Le Bodic and Nemhauser~\cite{LeBodic2017} is that other than the change in branching rules, we allow SCIP to use its default settings. This contrasts with the experiments of Le Bodic and Nemhauser, which provided problems with their primal bound, disabled primal heuristics, disabled cuts after the root node, and disabled the connected components presolver. Although these changes may lead to reduced variability \cite{lodi2013performance}, such setups often do not reflect the true performance of the solver on real world instances with real settings \cite{alvarez2017machine}. Additionally, the MIPLIB 2017 Benchmark Set has been designed to include problems that exhibit strong numerical stability, and hence less variability, for this reason.
  
  The summary results are depicted in Table~\ref{tab:rubberband_all}. The results are divided into three major columns, representing respectively the \texttt{product} (SCIP's default), \texttt{ratio}, and \texttt{svts} scoring functions. The three subcolumns of \texttt{product} show the number of instances that were solved, the time taken, and the search tree size respectively. For \texttt{ratio} and \texttt{svts}, these are measured relative to product. We report summaries in terms of the totals, the geometric means, and shifted geometric means (with shifts of $10$ and $100$ for time and nodes respectively, as is standard for MIP benchmarking \cite{achterberg2007thesis}). The best performing rule for each measurement is shown in bold. A table with per-instance statistics can be found in Appendix 3.
  
  \setlength{\tabcolsep}{0pt}
  
  \tablehead{
    \toprule
    Instance &  \multicolumn{3}{l}{\longunderline{\texttt{product}}} & \multicolumn{3}{l}{\longunderline{\texttt{ratio}}} & \multicolumn{3}{l}{\longunderline{\texttt{svts}}} \\
    &\# & Time & Nodes&\# & Time & Nodes&\# & Time & Nodes\\
    \midrule
  }
  \tabletail{
    \midrule
    \multicolumn{10}{r} \; continues on next page... \\
    \bottomrule
  }
  \tablelasttail{\bottomrule}
  
  \begin{table}[H]
    \centering
    {\small 
      \begin{supertabular*}{\linewidth}{@{\extracolsep{\fill}}l l l l l l l l l l}
        Total & 305 & 296.61k & 135.63m & +1 & 1.04 & 0.88 & \textbf{+3} & \textbf{0.96} & \textbf{0.82} \\
        Geo. mean & & 319.07 & \textbf{2.75k} &  & 1.01 & 1.08 &  & \textbf{0.99} & 1.02 \\
        Sh. geo. mean & &351.50 & \textbf{4.84k} &  & 1.02 & 1.05 &  & \textbf{0.99} & 1.00 \\
      \end{supertabular*}
    }
    \caption{Performance results on all instances of the MIPLIB 2017 Benchmark Set}\label{tab:rubberband_all}
  \end{table}
  
  \noindent Table~\ref{tab:rubberband_all} shows that both \texttt{ratio} and \texttt{svts} solve more instances than \texttt{product}, and require fewer nodes in total. However, the speedups are relatively small, with only a $1\%$ geometric average speedup for \texttt{svts}, and a $1\%-2\%$ slowdown for \texttt{ratio}. The larger improvements in arithmetic averages (i.e.\ totals) than geometric averages allude to the fact that the methods perform particularly well on instances requiring large B\&B trees. Since these scoring functions are indeed designed to work well on instances requiring large trees, we restrict our attention to the subset of instances for which at least one setting required at least $10,000$ and $50,000$ nodes. The summary results for these subsets of instances are shown in Tables \ref{tab:rubberband_somewhat_large} and \ref{tab:rubberband_large} respectively.
  
  \begin{table}[h]
    \centering
    {\small
      \begin{supertabular*}{\linewidth}{@{\extracolsep{\fill}}l l l l l l l l l l}
        Total & 141 & 216.35k & 135.49m & +4 & 1.01 & 0.88 & \textbf{+5} & \textbf{0.92} & \textbf{0.82} \\
        Geo. mean & & 858.30 & 92.19k &  & 0.99 & 1.04 &  & \textbf{0.93} & \textbf{0.94} \\
        Sh. geo. mean & &882.15 & 92.93k &  & 0.99 & 1.04 &  & \textbf{0.93} & \textbf{0.94} \\
      \end{supertabular*}
    }
    \caption{Performance results on instances of the MIPLIB 2017 Benchmark Set that required at least $10$k nodes for some setting}\label{tab:rubberband_somewhat_large}
  \end{table}
  
  \begin{table}[h]
    \centering
    {\small
      \begin{supertabular*}{\linewidth}{@{\extracolsep{\fill}}l l l l l l l l l l}
        Total & 93 & 144.30k & 134.87m & +4 & 1.00 & 0.87 & \textbf{+5} & \textbf{0.86} & \textbf{0.81} \\
        Geo. mean & & 1.02k & 326.96k &  & 0.96 & 1.02 &  & \textbf{0.89} & \textbf{0.89} \\
        Sh. geo. mean & &1.03k & 327.19k &  & 0.96 & 1.02 &  & \textbf{0.89} & \textbf{0.89} \\
      \end{supertabular*}
    }
    \caption{Performance results on instances of the MIPLIB 2017 Benchmark Set that required at least $50$k nodes for some setting}\label{tab:rubberband_large}
  \end{table}
  
  
On instances requiring at least $10,000$ nodes, \texttt{svts} achieves a $7\%$ and $6\%$ geometric average speedup and tree size reduction respectively, and solves more instances than \texttt{product}. On instances requiring $50,000$ nodes, this improves to $11\%$.
 The \texttt{ratio} scoring function performs less well, only outperforming product by $4\%$ in geometric average time, and not yielding smaller trees on average. These results confirm that \texttt{svts} is significantly better than \texttt{product} for MIPs that require very large B\&B trees.

	\section{Conclusions}
	
	In this paper, we resolved many of the open problems of Le Bodic and Nemhauser's B\&B model. We showed that there is no closed-form formula for the ratio value $\varphi$, and that the MVB conjecture is false. Additionally, we showed tighter hardness results for the GVB problem, and showed that the GVB simulation problem can be solved in expected sub-exponential time in $n$ for scoring functions that respect dominance. We then implemented improved branching rules for the MIP solver SCIP 6.0, which yielded an $11\%$ geometric average speedup and tree size reduction for problems in the MIPLIB 2017 Benchmark Set that required large B\&B trees.
  Since these rules perform well on MIP instances that lead to large B\&B trees, an interesting line of future work would be to incorporate the methods of Anderson et al.\ \cite{anderson2019clairvoyant} for predicting B\&B tree sizes to select branching rules at run time, or to reconsider the choice of branching rule after performing a restart.
  
  {\bigskip\noindent\textbf{Acknowledgments } This research was funded by a Monash Faculty of IT grant. We would like to thank Professor Graham Farr for introducing the authors to one another, hence without whom this research may have never happened. We also thank Gregor Hendel and the SCIP team for assistance with the performance tests. Finally, we are indebted to the referees for their thorough and helpful input, thanks to which the quality of this manuscript was substantially improved.}
  
	{
    \clearpage
  	\bibliographystyle{spmpsci}
  	\bibliography{ref}
  }

  \clearpage
  
  \section*{Appendix 1: Solution to the MVB counterexample}
  
  We prove the correctness of \eqref{eqn:closed_form_mvb} by induction. For $G \leq 5$, we can confirm exhaustively that $t(0) = 1,\ t(1) = 3,\ t(2) = 3,\ t(3) = 3,\ t(4) = 5,\ t(5) = 7$. Then, for $G \geq 6$, suppose that \eqref{eqn:closed_form_mvb} is a solution to the instance. We have
  
  \begin{multicols}{2}
    \small
    \begin{flalign*}
    &t(0 + 6k) \\
    &= 1 + \min \begin{cases}
    t(0 + 6k - 2) + t(0 + 6k - 4), \\
    t(0 + 6k - 3) + t(0 + 6k - 3), \\
    \end{cases}\\
    &= 1 + \min \begin{cases}
    t(4 + 6(k-1)) + t(2 + 6(k-1)),\\
    t(3 + 6(k-1)) + t(3 + 6(k-1)),\\
    \end{cases}\\
    &= 1 + \min \begin{cases}
    \frac{2}{3}\left( 2 \cdot 4^k + 1 \right) - 1 + \frac{1}{6}\left( 5 \cdot 4^k + 4 \right) - 1,\\
    4^k - 1 + 4^k - 1,\\
    \end{cases}\\
    &= \min \begin{cases}
    \frac{1}{6}\left( 13 \cdot 4^k + 8 \right) - 1,\\
    2 \cdot 4^k - 1,\\
    \end{cases}\\
    &= 2 \cdot 4^k - 1.
    \end{flalign*}
    
    \begin{flalign*}
    &t(1 + 6k) \\
    &= 1 + \min \begin{cases}
    t(1 + 6k - 2) + t(1 + 6k - 4), \\
    t(1 + 6k - 3) + t(1 + 6k - 3), \\
    \end{cases}\\
    &= 1 + \min \begin{cases}
    t(5 + 6(k-1)) + t(3 + 6(k-1)),\\
    t(4 + 6(k-1)) + t(4 + 6(k-1)),\\
    \end{cases}\\
    &= 1 + \min \begin{cases}
    \frac{1}{3}\left( 5 \cdot 4^k + 4 \right) - 1 + 4^k - 1,\\
    \frac{2}{3}\left( 2 \cdot 4^k + 1 \right) - 1 + \frac{2}{3}\left( 2 \cdot 4^k + 1 \right) - 1,\\
    \end{cases}\\
    &= \min \begin{cases}
    \frac{4}{3}\left( 2 \cdot 4^k + 1 \right) - 1,\\
    \frac{4}{3}\left( 2 \cdot 4^k + 1 \right) - 1,\\
    \end{cases}\\
    &= \frac{4}{3}\left( 2 \cdot 4^k + 1 \right) - 1.
    \end{flalign*}
    \begin{flalign*}
    &t(2 + 6k) \\
    &= 1 + \min \begin{cases}
    t(2 + 6k - 2) + t(2 + 6k - 4), \\
    t(2 + 6k - 3) + t(2 + 6k - 3), \\
    \end{cases}\\
    &= 1 + \min \begin{cases}
    t(0 + 6k) + t(4 + 6(k-1)),\\
    t(5 + 6(k-1)) + t(5 + 6(k-1)),\\
    \end{cases}\\
    &= 1 + \min \begin{cases} \label{eqn:mvb_counterexample_case}
    2 \cdot 4^k - 1 + \frac{2}{3}\left( 2 \cdot 4^k + 1 \right) - 1,\\
    \frac{1}{3}\left( 5 \cdot 4^k + 4 \right) - 1 + \frac{1}{3}\left( 5 \cdot 4^k + 4 \right) - 1,\\
    \end{cases}\\
    &= \min \begin{cases} 
    \frac{2}{3}\left( 5 \cdot 4^k + 1 \right) - 1,\\
    \frac{2}{3}\left( 5 \cdot 4^k + 4 \right) - 1,\\
    \end{cases}\\
    &= \frac{2}{3}\left( 5 \cdot 4^k + 1 \right) - 1.
    \end{flalign*}
    \begin{flalign*}
    &t(3 + 6k) \\
    &= 1 + \min \begin{cases}
    t(3 + 6k - 2) + t(3 + 6k - 4), \\
    t(3 + 6k - 3) + t(3 + 6k - 3), \\
    \end{cases}\\
    &= 1 + \min \begin{cases}
    t(1 + 6k) + t(5 + 6(k-1)),\\
    t(0 + 6k) + t(0 + 6k),\\
    \end{cases}\\
    &= 1 + \min \begin{cases}
    \frac{4}{3}\left( 2 \cdot 4^k + 1 \right) - 1 + \frac{1}{3}\left( 5 \cdot 4^k + 4 \right) - 1,\\
    2 \cdot 4^k - 1 + 2 \cdot 4^k - 1,\\
    \end{cases}\\
    &= \min \begin{cases}
    \frac{1}{3}\left( 13 \cdot 4^k + 8 \right) - 1,\\
    4 \cdot 4^k - 1,\\
    \end{cases}\\
    &= 4 \cdot 4^k - 1.
    \end{flalign*}
    \begin{flalign*}
    &t(4 + 6k) \\
    &= 1 + \min \begin{cases}
    t(4 + 6k - 2) + t(4 + 6k - 4), \\
    t(4 + 6k - 3) + t(4 + 6k - 3), \\
    \end{cases}\\
    &= 1 + \min \begin{cases}
    t(2 + 6k) + t(0 + 6k),\\
    t(1 + 6k) + t(1 + 6k),\\
    \end{cases}\\
    &= 1 + \min \begin{cases}
    \frac{2}{3}\left( 5 \cdot 4^k + 1 \right) - 1 + 2 \cdot 4^k - 1,\\
    \frac{4}{3}\left( 2 \cdot 4^k + 1 \right) - 1 + \frac{4}{3}\left( 2 \cdot 4^k + 1 \right) - 1,\\
    \end{cases}\\
    &= \min \begin{cases}
    \frac{2}{3}\left( 8 \cdot 4^k + 1 \right) - 1,\\
    \frac{8}{3}\left( 2 \cdot 4^k + 1 \right) - 1,\\
    \end{cases}\\
    &= \frac{2}{3}\left( 8 \cdot 4^k + 1 \right) - 1.
    \end{flalign*}
    \begin{flalign*}
    &t(5 + 6k) \\
    &= 1 + \min \begin{cases}
    t(5 + 6k - 2) + t(5 + 6k - 4), \\
    t(5 + 6k - 3) + t(5 + 6k - 3), \\
    \end{cases}\\
    &= 1 + \min \begin{cases}
    t(3 + 6k) + t(1 + 6k),\\
    t(2 + 6k) + t(2 + 6k),\\
    \end{cases}\\
    &= 1 + \min \begin{cases}
    4 \cdot 4^k - 1 + \frac{4}{3}\left( 2 \cdot 4^k + 1 \right) - 1,\\
    \frac{2}{3}\left( 5 \cdot 4^k + 1 \right) - 1 + \frac{2}{3}\left( 5 \cdot 4^k + 1 \right) - 1,\\
    \end{cases}\\
    &= \min \begin{cases}
    \frac{4}{3}\left( 5 \cdot 4^k + 1 \right) - 1,\\
    \frac{4}{3}\left( 5 \cdot 4^k + 1 \right) - 1,\\
    \end{cases}\\
    &= \frac{4}{3}\left( 5 \cdot 4^k + 1 \right) - 1.
    \end{flalign*}      
  \end{multicols}
  
  \noindent Therefore by induction on $G$, we can conclude that \eqref{eqn:closed_form_mvb} is a solution to the instance.
  
  \section*{Appendix 2: Evaluating scoring functions using the GVB simulation problem}\label{sec:preliminary_comp}
  
  We demonstrate the utility of the improved analysis of the GVB simulation problem as a tool for calibrating and predicting the performance of proposed variable selection rules.
  
  \subsection*{A practical algorithm for the GVB simulation problem}
  
  Using the fact that the expected number of non-dominated subsets of a set of random variables is sub-exponential, a dynamic programming algorithm that, at each step, filtered out the dominated variables and only generated states implicitly could solve \eqref{eqn:gvb_with_selection_rule} in expected sub-exponential time and space. We present here some techniques that lead to an even more practical algorithm.
  
  \begin{enumerate}[leftmargin=*]
    \item We first note that dominance is clearly transitive by definition, and that if $v_1$ dominates $v_2$, then $v_2$ can not dominate $v_1$. The dominance relation therefore defines a directed acyclic graph (DAG) on the variables. Note that the order that a non-dominating variable selection rule chooses to branch on a set of variables must therefore be a topological sort of the given DAG. For random data, the dominance DAG can be very dense, so, for efficiency, our algorithm computes the transitive reduction of the dominance DAG. The transitive reduction of the dominance DAG contains for each vertex, edges to those visible on the upper convex hull from the point of view of that vertex, hence the expected outdegree will be $O(\log(n))$. See Figure \ref{fig:reduced_dominance_dag}.
    
    \item The dynamic programming states are indexed by the current subset of available non-dominated variables and the current gap. In order to speed up indexing the states, the algorithm first pre-computes the set of all non-dominated subsets using heuristics (1) and (3) and a depth-first search through the state space.
    
    \item When a variable is used, the algorithm examines its successors in the reduced dominance DAG. For each successor, if the selected variable was its last dominator, that variable is no longer dominated and is added to the subset of available non-dominated variables. This can be maintained in constant time by keeping track of the current indegree for each variable ({\`a} la Kahn's topological sorting algorithm). Alternatively, if the multiplicities are all one and we use fewer than $w$ variables, where $w$ is the size of a machine word, we can pre-compute the adjacency matrix of the reverse dominance DAG and store the current subset as a bitmask. Checking whether a variable is the last remaining dominator is then achieved in one operation with a bitwise AND.
  \end{enumerate}
  
  \begin{figure}[h]
    \centering
    \includegraphics[scale=0.35]{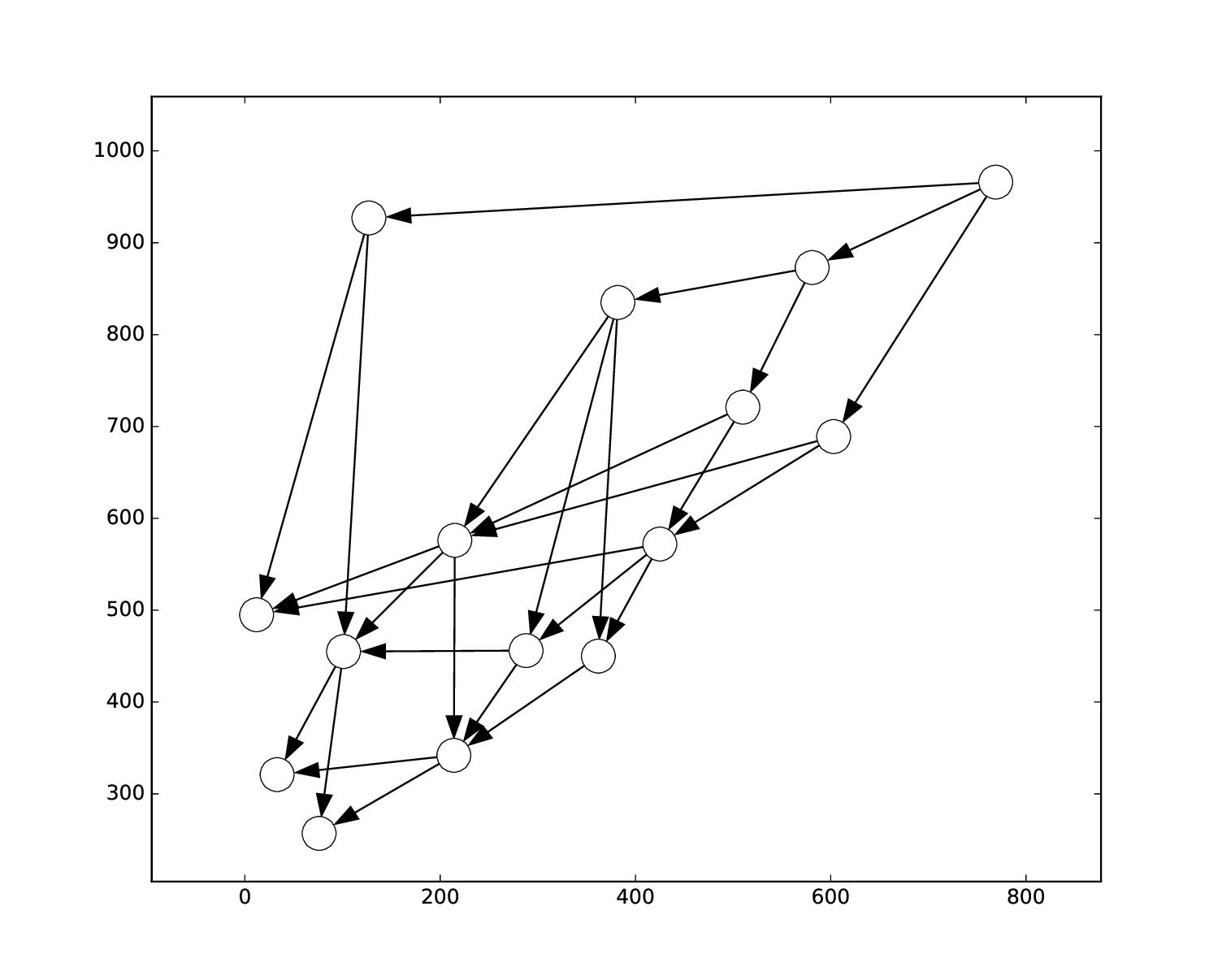}
    \caption{The transitive reduction of the dominance graph of 15 random variables in $[0,1000]^2$.}\label{fig:reduced_dominance_dag}
  \end{figure}
  
  \subsection*{Test parameters}
  
  For each test case, the algorithm generates $60$ random variables $(l,r)$. The test cases each fall into one of four categories:
  \begin{itemize}[leftmargin=12pt]
    \item \textbf{Balanced Instances}: The gains $l$ and $r$ are generated uniformly from the range $[1,1000]$. If $l > r$, then the two are switched to ensure $l \leq r$.
    \item \textbf{Unbalanced Instances}: The gains $l$ and $r$ are generated uniformly from the ranges $[1,500]$ and $[501,1000]$ respectively.
    \item \textbf{Very Unbalanced Instances}: The gains $l$ and $r$ are generated uniformly from the ranges $[1,250]$ and $[251,1000]$ respectively.
    \item \textbf{Extremely Unbalanced Instances}: The gains $l$ and $r$ are generated uniformly from the ranges $[1,125]$ and $[126,1000]$ respectively.
  \end{itemize}
  
  \noindent For simplicity, all variables have multiplicity $1$, so each leaf node of a resulting B\&B tree will have a depth at most $60$. We generate $3000$ test cases for each category, using three different values of $G$ to measure the effectiveness of the rules as the trees become larger. For the four categories above, we use the gap values $G$ depicted in Table \ref{tab:gvb_gaps}. Note that the gap decreases as the variables become less balanced since the trees will require an intractable number of nodes otherwise.
  
  \begin{table}[H]
    \centering
    \begingroup
    \def\arraystretch{1.1}
    \begin{tabular}{p{4.5cm} p{1.5cm} p{1.5cm} p{1.5cm}}
      \toprule
      Category & Small & Medium & Large \\
      \midrule
      Balanced & 5000 & 9000 & 12000 \\
      Unbalanced & 4000 & 7000 & 9000\\
      Very Unbalanced & 3000 & 5000 & 6000 \\
      Extremely Unbalanced & 2000 & 3000 & 3500 \\
      \bottomrule
    \end{tabular}
    \vspace{8pt}
    \caption{The gap sizes $G$ used in each of the computational tests.}\label{tab:gvb_gaps}
    \endgroup
  \end{table}
  
  \noindent The gaps in Table~\ref{tab:gvb_gaps} are chosen slightly differently to Le Bodic and Nemhauser \cite{LeBodic2017} in order to ensure that the simulations remain computationally feasible.
  
  \subsection*{GVB simulation experiments}
  
  Le Bodic and Nemhauser \cite{LeBodic2017} used the GVB simulation problem to tune the hybrid \texttt{ratio} rule, which selects from either the ratio or the product score depending on the estimated tree height. Due to the intractability of the GVB problem, they were unable to perform experiments on the \texttt{svts} rule. In this section, we use our improved GVB simulation algorithm to compare the \texttt{product} rule, the \texttt{ratio} rule and \texttt{svts}.
  
  Table~\ref{tab:hybrid_gvb_simulations} shows for each category and gap combination, for each scoring function, the estimated geometric average tree size relative to product (hence product shows all zeros).
  
  \begin{table*}[h!]
    \centering
    \begingroup
    \renewcommand*{\arraystretch}{1.1}
    \begin{tabular}{p{4.5cm} p{1.5cm} p{1.6cm} p{1.5cm} p{1.5cm} }
      \toprule
      Category & Gap & \texttt{product} & {\texttt{ratio}} & \texttt{svts} \\
      \midrule
      {Balanced} & 5000  & 0.00 &	0.00 &	\textbf{-5.69}	\\
      & 9000  & 0.00 &	0.00 &	\textbf{-8.87}  \\
      & 12000  & 0.00 &	0.00 &	\textbf{-9.54} \\
      \midrule
      {Unbalanced} & 4000  &  0.00 &	0.00 &	\textbf{-7.05} \\
      & 7000  & 0.00	& -0.00 &	\textbf{-9.14}  \\
      & 9000  &  0.00 &	-0.01 &	\textbf{-9.89}  \\
      \midrule
      {Very \mbox{unbalanced}} & 3000  &  0.00 &	-0.09 &	\textbf{-10.98} \\
      & 5000  & 0.00 &	-1.17 &	\textbf{-14.12} \\
      & 6000  & 0.00 &	-2.47 &	\textbf{-15.44}  \\
      \midrule
      {Extremely \mbox{unbalanced}} & 2000 &  0.00 &	-2.89 &	\textbf{-13.63}	 \\
      & 3000  & 0.00 &	-8.42 &	\textbf{-18.38} \\
      & 3500  & 0.00 &	-11.12 &	\textbf{-20.68} \\
      \bottomrule
    \end{tabular}
    \vspace{8pt}
    \caption{Results of the GVB simulation problem. The best performing rule for each category and gap combination is shown in bold. The \texttt{ratio} column refers to the hybrid \texttt{ratio-product} rule.}\label{tab:hybrid_gvb_simulations}
    \endgroup
  \end{table*}
  
  \noindent The results depicted in Table~\ref{tab:hybrid_gvb_simulations} clearly predict that the \texttt{svts} scoring function will outperform the hybrid \texttt{ratio} function, yielding tree size reductions between $5\%$ and $20\%$. This backs up the results of our MIP benchmarks, in which \texttt{svts} does indeed outperform \texttt{ratio} on average.
  
  \section*{Appendix 3: Full performance test results}
  
  Here we present complete data on the performance tests of the MIPLIB 2017 Benchmark Set. For each row (i.e.\ benchmark problem), we report the geometric average time and tree size of the best $N$ solved instances for each scoring function, where $N$ is the minimum number of instances of that problem solved by any scoring function. When a rule fails to solve any instance of a problem, time and node comparisons for that row are omitted, but number of instances solved is still depicted. As in the summary results, measurements for \texttt{ratio} and \texttt{svts} are given relative to \texttt{product}.
  
  \tablehead{
    \toprule
    Instance &  \multicolumn{3}{l}{\longunderline{\texttt{product}}} & \multicolumn{3}{l}{\longunderline{\texttt{ratio}}} & \multicolumn{3}{l}{\longunderline{\texttt{svts}}} \\
    &\# & Time & Nodes&\# & Time & Nodes&\# & Time & Nodes\\
    \midrule
  }
  \tabletail{
    \midrule
    \multicolumn{10}{r} \; continues on next page... \\
    \bottomrule
  }
  
  
  {
    \bigskip
    \filbreak 
    \small
    \bottomcaption{Detailed performance results on all instances of the MIPLIB 2017 Benchmark Set. The best performing rule for each problem is shown in bold.}\label{tab:fulll_all_results}
    \begin{supertabular*}{\linewidth}{@{\extracolsep{\fill}}l l l l l l l l l l}
        30n20b8 &\textbf{3} & 281.40 & 138.00 & \textbf{+0} & 0.92 & 1.26 & \textbf{+0} & \textbf{0.70} & \textbf{0.56} \\
        CMS750\_4 &\textbf{3} & \textbf{857.93} & \textbf{11.91k} & \textbf{+0} & 1.25 & 1.03 & \textbf{+0} & 2.15 & 2.06 \\
        air05 &\textbf{3} & \textbf{30.10} & \textbf{333.67} & \textbf{+0} & 1.20 & 1.91 & \textbf{+0} & 1.17 & 1.19 \\
        app1-1 &\textbf{3} & 5.97 & \textbf{4.00} & \textbf{+0} & 1.01 & \textbf{1.00} & \textbf{+0} & \textbf{0.99} & \textbf{1.00} \\
        app1-2 &\textbf{3} & 1.04k & 68.33 & \textbf{+0} & \textbf{0.83} & \textbf{0.62} & \textbf{+0} & 1.28 & 0.85 \\
        assign1-5-8 &\textbf{3} & 3.35k & 5.94m & \textbf{+0} & 0.96 & 0.93 & \textbf{+0} & \textbf{0.85} & \textbf{0.83} \\
        beasleyC3 &\textbf{3} & 22.57 & \textbf{2.00} & \textbf{+0} & 1.00 & \textbf{1.00} & \textbf{+0} & \textbf{1.00} & \textbf{1.00} \\
        binkar10\_1 &\textbf{3} & 28.47 & 2.58k & \textbf{+0} & \textbf{0.79} & 0.88 & \textbf{+0} & 0.83 & \textbf{0.86} \\
        bnatt400 &\textbf{3} & 1.18k & \textbf{6.85k} & \textbf{+0} & 0.96 & 1.07 & \textbf{+0} & \textbf{0.96} & 1.07 \\
        bnatt500 &\textbf{3} & 5.02k & 30.07k & \textbf{+0} & 0.97 & \textbf{0.96} & \textbf{+0} & \textbf{0.97} & \textbf{0.96} \\
        bppc4-08 &0 & - & - & \textbf{+2} & - & - & \textbf{+2} & - & - \\
        brazil3 &\textbf{3} & 4.01k & \textbf{1.69k} & \textbf{+0} & \textbf{1.00} & 1.47 & \textbf{+0} & 1.04 & 1.39 \\
        chromaticindex512-7 &\textbf{3} & \textbf{2.21k} & \textbf{5.89k} & \textbf{+0} & 1.01 & 1.24 & \textbf{+0} & 1.01 & 1.24 \\
        co-100 &\textbf{2} & 4.53k & \textbf{6.11k} & \textbf{+0} & \textbf{0.93} & 1.64 & \textbf{+0} & 1.09 & 2.01 \\
        cod105 &\textbf{3} & \textbf{315.37} & \textbf{105.00} & \textbf{+0} & 1.16 & 2.05 & \textbf{+0} & 1.18 & 2.04 \\
        cost266-UUE &\textbf{3} & 3.18k & 234.69k & \textbf{+0} & 1.34 & 1.46 & \textbf{+0} & \textbf{0.96} & \textbf{0.95} \\
        csched007 &\textbf{3} & 3.11k & 266.24k & \textbf{+0} & \textbf{0.53} & \textbf{0.51} & \textbf{+0} & 0.71 & 0.57 \\
        csched008 &\textbf{3} & \textbf{1.03k} & \textbf{94.19k} & \textbf{+0} & 1.27 & 1.24 & \textbf{+0} & 1.36 & 1.42 \\
        dano3\_3 &\textbf{3} & \textbf{106.70} & \textbf{12.33} & \textbf{+0} & 1.10 & 1.68 & \textbf{+0} & 1.02 & 1.41 \\
        dano3\_5 &\textbf{3} & 307.80 & 165.00 & \textbf{+0} & \textbf{0.86} & \textbf{0.89} & \textbf{+0} & 0.89 & 1.04 \\
        drayage-100-23 &\textbf{3} & 16.83 & 33.33 & \textbf{+0} & 0.95 & 1.03 & \textbf{+0} & \textbf{0.89} & \textbf{0.90} \\
        drayage-25-23 &\textbf{3} & 1.29k & \textbf{106.45k} & -2 & 4.99 & 18.14 & -2 & \textbf{0.97} & 2.67 \\
        eil33-2 &\textbf{3} & 70.93 & 797.00 & \textbf{+0} & \textbf{0.89} & \textbf{0.91} & \textbf{+0} & 0.97 & 0.97 \\
        fast0507 &\textbf{3} & \textbf{217.13} & \textbf{840.00} & \textbf{+0} & 1.44 & 1.54 & \textbf{+0} & 1.20 & 1.15 \\
        fastxgemm-n2r6s0t2 &\textbf{3} & 605.93 & 113.70k & \textbf{+0} & 1.35 & 1.42 & \textbf{+0} & \textbf{0.76} & \textbf{0.73} \\
        fiball &1 & 1.47k & 3.56k & \textbf{+1} & \textbf{0.89} & \textbf{0.80} & \textbf{+1} & 1.17 & 1.16 \\
        gen-ip002 &\textbf{3} & 1.55k & 5.59m & \textbf{+0} & \textbf{0.91} & \textbf{0.93} & \textbf{+0} & 0.96 & 1.00 \\
        gen-ip054 &\textbf{3} & 3.04k & 12.77m & \textbf{+0} & \textbf{0.63} & \textbf{0.62} & \textbf{+0} & 0.63 & 0.63 \\
        glass-sc &\textbf{3} & 3.16k & 278.80k & \textbf{+0} & \textbf{0.93} & \textbf{0.96} & \textbf{+0} & 0.95 & 0.98 \\
        glass4 &\textbf{3} & \textbf{2.17k} & \textbf{1.73m} & \textbf{+0} & 1.54 & 2.37 & \textbf{+0} & 1.14 & 1.63 \\
        graph20-20-1rand &2 & 5.17k & 38.68k & +0 & \textbf{0.34} & \textbf{0.37} & \textbf{+1} & 0.49 & 0.49 \\
        graphdraw-domain &\textbf{3} & \textbf{1.39k} & \textbf{2.52m} & \textbf{+0} & 1.15 & 1.11 & \textbf{+0} & 1.03 & 1.01 \\
        h80x6320d &\textbf{3} & \textbf{105.67} & \textbf{4.00} & \textbf{+0} & 1.00 & \textbf{1.00} & \textbf{+0} & 1.00 & \textbf{1.00} \\
        icir97\_tension &0 & - & - & \textbf{+3} & - & - & +1 & - & - \\
        irp &\textbf{3} & \textbf{12.90} & \textbf{5.33} & \textbf{+0} & 1.01 & \textbf{1.00} & \textbf{+0} & 1.00 & \textbf{1.00} \\
        istanbul-no-cutoff &\textbf{3} & 179.57 & 302.33 & \textbf{+0} & \textbf{0.95} & 1.03 & \textbf{+0} & 1.02 & \textbf{0.99} \\
        map10 &\textbf{3} & \textbf{775.90} & \textbf{1.15k} & \textbf{+0} & 1.01 & 1.32 & \textbf{+0} & 1.11 & 1.43 \\
        map16715-04 &\textbf{3} & \textbf{1.75k} & \textbf{1.73k} & \textbf{+0} & 1.17 & 1.35 & \textbf{+0} & 1.21 & 1.31 \\
        markshare\_4\_0 &\textbf{3} & 320.37 & 2.51m & \textbf{+0} & \textbf{0.78} & \textbf{0.79} & \textbf{+0} & 0.87 & 0.88 \\
        mas74 &\textbf{3} & 2.29k & 7.07m & \textbf{+0} & 0.89 & 0.84 & \textbf{+0} & \textbf{0.77} & \textbf{0.78} \\
        mas76 &\textbf{3} & 132.87 & 301.16k & \textbf{+0} & 0.98 & 1.02 & \textbf{+0} & \textbf{0.90} & \textbf{0.88} \\
        mc11 &\textbf{3} & \textbf{112.33} & 2.26k & \textbf{+0} & 1.01 & \textbf{0.95} & \textbf{+0} & 1.55 & 2.12 \\
        mcsched &\textbf{3} & 240.10 & \textbf{11.39k} & \textbf{+0} & \textbf{0.99} & 1.31 & \textbf{+0} & 1.00 & 1.08 \\
        mik-250-20-75-4 &\textbf{3} & 37.23 & 16.05k & \textbf{+0} & \textbf{0.74} & \textbf{0.73} & \textbf{+0} & 0.79 & 0.76 \\
        mzzv11 &\textbf{3} & 336.90 & 1.57k & \textbf{+0} & \textbf{0.90} & 0.98 & \textbf{+0} & 0.94 & \textbf{0.80} \\
        mzzv42z &\textbf{3} & \textbf{168.13} & \textbf{259.67} & \textbf{+0} & 1.26 & 1.22 & \textbf{+0} & 1.19 & 1.34 \\
        n2seq36q &\textbf{3} & 921.67 & 4.12k & \textbf{+0} & \textbf{0.94} & \textbf{0.52} & \textbf{+0} & 0.95 & 0.86 \\
        n5-3 &\textbf{3} & \textbf{30.50} & \textbf{785.33} & \textbf{+0} & 1.07 & 1.17 & \textbf{+0} & 1.06 & 1.22 \\
        neos-1445765 &\textbf{3} & \textbf{59.53} & \textbf{72.33} & \textbf{+0} & 1.03 & 1.29 & \textbf{+0} & 1.01 & \textbf{1.00} \\
        neos-1456979 &\textbf{1} & - & - & -1 & - & - & -1 & - & - \\
        neos-1582420 &\textbf{3} & 37.30 & \textbf{664.33} & \textbf{+0} & \textbf{0.99} & 1.02 & \textbf{+0} & 1.21 & 1.67 \\
        neos-2657525-crna &0 & - & - & +0 & - & - & \textbf{+1} & - & - \\
        neos-3004026-krka &\textbf{3} & 68.13 & 5.25k & \textbf{+0} & \textbf{0.82} & \textbf{0.66} & \textbf{+0} & 0.82 & \textbf{0.66} \\
        neos-3024952-loue &2 & 1.22k & 135.87k & \textbf{+1} & 1.25 & 0.92 & \textbf{+1} & \textbf{0.99} & \textbf{0.74} \\
        neos-3083819-nubu &\textbf{3} & 14.93 & 2.33k & \textbf{+0} & \textbf{0.86} & 1.01 & \textbf{+0} & 0.89 & \textbf{1.00} \\
        neos-3216931-puriri &\textbf{1} & - & - & -1 & - & - & -1 & - & - \\
        neos-3402294-bobin &\textbf{3} & 1.84k & 12.76k & \textbf{+0} & 1.02 & 1.44 & \textbf{+0} & \textbf{0.68} & \textbf{0.53} \\
        neos-3627168-kasai &0 & - & - & \textbf{+1} & - & - & \textbf{+1} & - & - \\
        neos-4413714-turia &\textbf{3} & 387.20 & \textbf{2.00} & \textbf{+0} & \textbf{0.99} & \textbf{1.00} & \textbf{+0} & 1.00 & \textbf{1.00} \\
        neos-4722843-widden &\textbf{3} & \textbf{1.60k} & \textbf{2.62k} & \textbf{+0} & 1.12 & 1.15 & \textbf{+0} & 1.01 & 1.01 \\
        neos-4738912-atrato &\textbf{3} & 1.02k & 103.05k & \textbf{+0} & \textbf{0.71} & \textbf{0.60} & \textbf{+0} & 0.87 & 0.82 \\
        neos-5107597-kakapo &\textbf{3} & \textbf{2.10k} & \textbf{682.60k} & -2 & 2.67 & 3.52 & -1 & 1.85 & 2.23 \\
        neos-5188808-nattai &\textbf{3} & 2.22k & 27.40k & \textbf{+0} & \textbf{0.83} & \textbf{0.69} & \textbf{+0} & 0.84 & 0.81 \\
        neos-5195221-niemur &\textbf{3} & \textbf{2.68k} & \textbf{21.97k} & \textbf{+0} & 1.10 & 1.09 & \textbf{+0} & 1.25 & 1.31 \\
        neos-848589 &\textbf{1} & - & - & \textbf{+0} & - & - & -1 & - & - \\
        neos-860300 &\textbf{3} & \textbf{16.73} & \textbf{2.00} & \textbf{+0} & \textbf{1.00} & \textbf{1.00} & \textbf{+0} & 1.01 & \textbf{1.00} \\
        neos-911970 &\textbf{2} & \textbf{1.12k} & \textbf{623.61k} & -1 & 2.55 & 2.37 & -1 & 2.32 & 1.27 \\
        neos-933966 &\textbf{3} & 1.87k & 582.00 & \textbf{+0} & 1.60 & \textbf{0.91} & \textbf{+0} & \textbf{0.93} & 1.03 \\
        neos-950242 &\textbf{3} & 283.60 & 71.33 & \textbf{+0} & \textbf{0.82} & \textbf{0.72} & \textbf{+0} & 0.85 & 0.74 \\
        neos-960392 &\textbf{3} & \textbf{1.22k} & \textbf{65.67} & \textbf{+0} & 1.09 & 1.04 & \textbf{+0} & 1.09 & 1.04 \\
        neos17 &\textbf{3} & 22.07 & \textbf{16.59k} & \textbf{+0} & 1.17 & 1.06 & \textbf{+0} & \textbf{1.00} & 1.03 \\
        neos5 &\textbf{3} & 104.37 & 255.57k & \textbf{+0} & \textbf{0.92} & \textbf{0.96} & \textbf{+0} & 1.12 & 1.19 \\
        net12 &\textbf{3} & 1.34k & 2.07k & \textbf{+0} & \textbf{0.71} & \textbf{0.85} & \textbf{+0} & 0.73 & 0.88 \\
        netdiversion &\textbf{3} & \textbf{951.75} & \textbf{13.00} & -1 & 1.66 & 2.08 & \textbf{+0} & 2.30 & 3.96 \\
        nexp-150-20-8-5 &\textbf{2} & 2.35k & \textbf{3.60k} & \textbf{+0} & 1.34 & 2.18 & \textbf{+0} & \textbf{0.92} & 2.09 \\
        ns1208400 &\textbf{3} & 513.60 & \textbf{1.26k} & \textbf{+0} & 0.79 & 1.21 & \textbf{+0} & \textbf{0.79} & 1.21 \\
        ns1644855 &\textbf{2} & 2.09k & \textbf{8.00} & -1 & 1.01 & \textbf{1.00} & -1 & \textbf{0.99} & \textbf{1.00} \\
        ns1830653 &\textbf{3} & \textbf{115.40} & 6.23k & \textbf{+0} & 1.24 & 1.20 & \textbf{+0} & 1.23 & \textbf{0.93} \\
        ns1952667 &\textbf{3} & 1.56k & \textbf{3.19k} & \textbf{+0} & \textbf{0.53} & 1.16 & \textbf{+0} & 0.53 & 1.16 \\
        nu25-pr12 &\textbf{3} & 6.60 & 97.00 & \textbf{+0} & \textbf{0.99} & \textbf{0.88} & \textbf{+0} & 1.04 & 0.97 \\
        nursesched-sprint02 &\textbf{3} & \textbf{39.87} & 13.00 & \textbf{+0} & 1.02 & \textbf{0.97} & \textbf{+0} & 1.01 & \textbf{0.97} \\
        nw04 &\textbf{3} & \textbf{23.87} & \textbf{7.33} & \textbf{+0} & 1.01 & \textbf{1.00} & \textbf{+0} & 1.00 & \textbf{1.00} \\
        p200x1188c &\textbf{3} & 2.90 & \textbf{3.00} & \textbf{+0} & \textbf{0.99} & \textbf{1.00} & \textbf{+0} & \textbf{0.99} & \textbf{1.00} \\
        peg-solitaire-a3 &\textbf{1} & \textbf{4.09k} & \textbf{1.92k} & \textbf{+0} & 1.53 & 2.02 & \textbf{+0} & 1.52 & 2.02 \\
        pg &\textbf{3} & 20.47 & 552.67 & \textbf{+0} & \textbf{0.99} & 0.94 & \textbf{+0} & 0.99 & \textbf{0.92} \\
        pg5\_34 &\textbf{3} & 2.09k & 191.13k & \textbf{+0} & 0.84 & 0.98 & \textbf{+0} & \textbf{0.72} & \textbf{0.90} \\
        physiciansched6-2 &\textbf{1} & \textbf{165.20} & \textbf{151.00} & \textbf{+0} & 2.15 & 12.62 & \textbf{+0} & 1.52 & 7.28 \\
        piperout-08 &\textbf{3} & 858.30 & 604.00 & \textbf{+0} & \textbf{0.93} & 0.81 & \textbf{+0} & 1.05 & \textbf{0.62} \\
        piperout-27 &\textbf{3} & \textbf{320.93} & \textbf{174.00} & \textbf{+0} & 1.65 & 2.13 & \textbf{+0} & 1.49 & 1.17 \\
        pk1 &\textbf{3} & 149.93 & 386.87k & \textbf{+0} & 1.03 & 1.06 & \textbf{+0} & \textbf{0.96} & \textbf{0.96} \\
        qap10 &\textbf{3} & 108.17 & \textbf{3.33} & \textbf{+0} & \textbf{0.99} & \textbf{1.00} & \textbf{+0} & 1.03 & \textbf{1.00} \\
        rail507 &\textbf{3} & 255.57 & 1.32k & \textbf{+0} & 1.03 & 0.80 & \textbf{+0} & \textbf{0.79} & \textbf{0.64} \\
        ran14x18-disj-8 &\textbf{3} & 1.52k & 446.14k & \textbf{+0} & \textbf{0.67} & \textbf{0.55} & \textbf{+0} & 0.78 & 0.64 \\
        rd-rplusc-21 &0 & - & - & \textbf{+1} & - & - & +0 & - & - \\
        reblock115 &1 & 5.77k & \textbf{802.39k} & \textbf{+1} & \textbf{0.87} & 1.04 & +0 & 1.07 & 1.07 \\
        rmatr100-p10 &\textbf{3} & 174.13 & 857.67 & \textbf{+0} & 1.04 & 1.11 & \textbf{+0} & \textbf{1.00} & \textbf{0.97} \\
        rocI-4-11 &\textbf{3} & \textbf{727.85} & \textbf{13.22k} & -1 & 5.34 & 8.54 & \textbf{+0} & 2.51 & 2.56 \\
        rococoC10-001000 &\textbf{3} & 985.13 & 74.65k & \textbf{+0} & 0.89 & 1.30 & \textbf{+0} & \textbf{0.52} & \textbf{0.53} \\
        roi2alpha3n4 &\textbf{3} & \textbf{844.00} & \textbf{5.64k} & \textbf{+0} & 1.51 & 1.97 & \textbf{+0} & 1.21 & 1.30 \\
        roll3000 &\textbf{3} & 46.90 & 2.59k & \textbf{+0} & \textbf{0.75} & \textbf{0.61} & \textbf{+0} & 1.05 & 1.02 \\
        s250r10 &2 & 3.00k & 28.38k & \textbf{+1} & 1.41 & 1.58 & \textbf{+1} & \textbf{0.92} & \textbf{0.45} \\
        seymour1 &\textbf{3} & 60.80 & 1.49k & \textbf{+0} & \textbf{0.91} & \textbf{0.74} & \textbf{+0} & 1.00 & 0.77 \\
        sp98ar &0 & - & - & +0 & - & - & \textbf{+1} & - & - \\
        supportcase18 &\textbf{1} & - & - & -1 & - & - & -1 & - & - \\
        supportcase26 &1 & 6.16k & 7.60m & \textbf{+2} & \textbf{0.20} & \textbf{0.22} & +1 & 0.28 & 0.30 \\
        supportcase33 &\textbf{3} & 1.67k & 15.08k & \textbf{+0} & 1.23 & 1.20 & \textbf{+0} & \textbf{0.82} & \textbf{0.69} \\
        supportcase40 &\textbf{3} & 1.11k & 12.08k & \textbf{+0} & 1.05 & 1.05 & \textbf{+0} & \textbf{0.98} & \textbf{0.97} \\
        supportcase7 &\textbf{3} & 177.13 & \textbf{28.33} & \textbf{+0} & \textbf{1.00} & 1.05 & \textbf{+0} & 1.00 & 1.06 \\
        swath1 &\textbf{3} & 14.93 & 387.00 & \textbf{+0} & 1.03 & 1.05 & \textbf{+0} & \textbf{0.99} & \textbf{0.94} \\
        swath3 &\textbf{3} & 425.53 & 72.00k & \textbf{+0} & \textbf{0.33} & 0.32 & \textbf{+0} & 0.38 & \textbf{0.30} \\
        tbfp-network &\textbf{3} & \textbf{1.02k} & \textbf{51.00} & -1 & 3.36 & 6.81 & \textbf{+0} & 1.84 & 3.28 \\
        timtab1 &\textbf{3} & \textbf{52.40} & \textbf{39.40k} & \textbf{+0} & 1.22 & 1.33 & \textbf{+0} & 1.25 & 1.33 \\
        tr12-30 &\textbf{3} & 832.57 & 501.01k & \textbf{+0} & \textbf{0.73} & \textbf{0.72} & \textbf{+0} & 1.04 & 1.00 \\
        uct-subprob &\textbf{3} & 2.50k & 94.91k & \textbf{+0} & 0.78 & 0.81 & \textbf{+0} & \textbf{0.76} & \textbf{0.73} \\
        unitcal\_7 &\textbf{3} & \textbf{280.30} & 180.33 & \textbf{+0} & 1.02 & 0.96 & \textbf{+0} & 1.00 & \textbf{0.60} \\
        var-smallemery-m6j6 &0 & - & - & +0 & - & - & \textbf{+1} & - & - \\
        wachplan &\textbf{3} & \textbf{857.33} & \textbf{45.68k} & \textbf{+0} & 1.14 & 1.16 & \textbf{+0} & 1.13 & 1.10 \\
    \end{supertabular*}
  }

\end{document}